\documentclass[a4paper]{article}

\usepackage[dvips]{graphics,graphicx}
\usepackage{amsfonts,amssymb,amsmath,color,mathrsfs, amstext}
\usepackage{amsbsy, amsopn, amscd, amsxtra, amsthm,authblk,enumerate,algorithmicx,algorithm}
\usepackage{bm,bbm}
\usepackage{algpseudocode}
\usepackage{upref}
\usepackage{geometry}
\usepackage{ulem}
\usepackage[displaymath]{lineno}
\usepackage{float}
\usepackage{yhmath}
\usepackage[colorlinks,
            linkcolor=red,
            anchorcolor=red,
            citecolor=red
            ]{hyperref}

\numberwithin{equation}{section}
\graphicspath{fig}
\newtheorem{thm}{Theorem}[section]

\newtheorem{prop}[thm]{Proposition}

\newtheorem{exmp}[thm]{Example}

\numberwithin{equation}{section}

\DeclareMathOperator{\erfc}{erfc}

\newcommand{\abs}[1]{\left\vert#1\right\vert}
\newcommand{\eps}{\varepsilon}
\newcommand{\dint}{\displaystyle\int}

\begin{document}

\title{\bf{Random batch particle methods for the homogeneous Landau equation}}

\author[1]{Jos\'{e} Antonio Carrillo \thanks{carrillo@maths.ox.ac.uk}}
\author[2]{Shi Jin \thanks{shijin-m@sjtu.edu.cn}}
\author[3]{Yijia Tang \thanks{yijia\underline{~}tang@sjtu.edu.cn}}
\affil[1]{Mathematical Institute, University of Oxford, Oxford, OX2 6GG, UK.}
\affil[2]{School of Mathematical Sciences, Institute of Natural Sciences, MOE-LSC, Shanghai Jiao Tong University, Shanghai, 200240, P. R. China.}
\affil[3]{School of Mathematical Sciences, Shanghai Jiao Tong University, Shanghai, 200240, P. R. China.}

\date{}
\maketitle

\begin{abstract}
We consider in this paper random batch particle methods for efficiently solving the homogeneous Landau equation in plasma physics. The methods are stochastic variations of the particle methods proposed by Carrillo et al. [J. Comput. Phys.: X 7: 100066, 2020] using the random batch strategy. The collisions only take place inside the small but randomly selected batches so that the computational cost is reduced to $O(N)$ per time step. Meanwhile, our methods can preserve the conservation of mass, momentum, energy and the decay of entropy. Several numerical examples are performed to validate our methods.
\end{abstract}


\section{Introduction}
\label{sec: intro}

The Fokker-Planck-Landau equation, originally derived by Landau \cite{Landau1936}, is a fundamental integro-differential equation describing the evolution of the distribution for charged particles in plasma physics \cite{LifshitzPitaevskii}.  It models the binary collisions between charged particles with long-range Coulomb interaction, which is the grazing limit of the Boltzmann equation \cite{DegondDesreux92,Desvillettes92,Villani98}. Denote by $f(t,x,v)$ the mass distribution of charged particles at time $t$, position $x$ with velocity $v$, the Fokker-Planck-Landau equation reads
\begin{equation}
\partial_t f +v\cdot \nabla_x f+E\cdot\nabla_v f = Q(f,f):=\nabla_v \cdot \left(\int_{\mathbb{R}^d}A(v-v_*)\left(f(v_*)\nabla_v f(v)-f(v)\nabla_{v_*}f(v_*)\right)\mathrm{d}v_*\right),
\label{eqn: FPL}
\end{equation}
where $E$ is either an external force or a self-consistent force. The collision kernel
$$
A(z)=\Lambda|z|^{\gamma}(|z|^2I_d - z\otimes z), \quad -d-1\leq \gamma\leq 1,\quad \Lambda> 0, \quad d\geq 2,
$$
is symmetric positive semi-definite, $A(z)=A(-z), \ker A(z)=\mathbb{R}z$. Similar to the Boltzmann equation, $\gamma>0$, $\gamma=0$, $\gamma<0$ represents the hard potential, Maxwell molecules and the soft potential case respectively. The Coulomb potential where $d=3, \gamma=-3$ is of great significance since it is relevant in physical plasma applications \cite{DegondDesreux92}. 

In the numerical aspect of equation \eqref{eqn: FPL}, the approximation of the nonlocal quadratic Landau collision operator $Q(f,f)$ is a major difficulty. In this paper, we only focus on the spatially homogeneous Landau equation
\begin{equation}
	\partial_t f= Q(f,f).
	\label{eqn: homogeneous Landau}
\end{equation} 
It is well-known that equation \eqref{eqn: homogeneous Landau} has conservation of mass, momentum and energy since $\int_{\mathbb{R}^d}Q(f,f)(1,v,|v|^2)\mathrm{d}v=\bm{0}$. The Boltzmann entropy $$E(f)=\int_{\mathbb{R}^d}f\log f \mathrm{d}v$$ is dissipated through
$$
	\frac{\mathrm{d}E}{\mathrm{d}t}=-D=-\frac{1}{2}\iint_{\mathbb{R}^{2d}}B_{v,v_*}\cdot A(v-v_*)B_{v,v_*}ff_*\mathrm{d}v\mathrm{d}v_*\leq 0.
	$$ 
Here, 
$$
B_{v,v_*}=\nabla_v \frac{\delta E}{\delta f} - \nabla_{v_*} \frac{\delta E_*}{\delta f_*},
\quad \frac{\delta E}{\delta f}=\log f+1,
$$ 
and $f_*=f(v_*), E_*=E(f_*)$ for short.
Moreover, $f$ is the equilibrium of \eqref{eqn: homogeneous Landau} if and only if $f$ is given by the local Maxwellian
\begin{equation}
	\mathcal{M}_{\rho,u,T}=\frac{\rho}{(2\pi T)^{d/2}}\exp \left(-\frac{|v-u|^2}{2T}\right)
	\label{eqn: Maxwellian}
\end{equation}
with conserved quantities 
$$
\rho=\int_{\mathbb{R}^{d}} f \mathrm{d}v,\quad 
\rho u=\int_{\mathbb{R}^{d}} v f \mathrm{d}v, \quad 
\rho u^2+\rho dT=\int_{\mathbb{R}^{d}} |v|^2f \mathrm{d}v.
$$
The well-posedness of the homogeneous Landau equation in the hard potential \cite{DesvillettesVillani2000,DesvillettesVillani2000b} and Maxwell molecule \cite{Villani1998} cases are well-understood resorting to the notion of $H-$solution proposed by Villani \cite{Villani98}. In the soft potential case, there are still lots of open questions. It has been partially resolved by probabilistic techniques \cite{FournierGuerin09} and entropy dissipation estimate \cite{Desvillettes15}.

Various numerical methods have been proposed for the homogeneous Landau equation. The entropy schemes \cite{DegondDesreux94,BuetCordier98} are widely used to provide a decay of numerical entropy. Thus, they can preserve the conserved quantities and the stationary states are discrete Maxwellians. And there are conservative implicit schemes \cite{LemouMieussens05} as well. To overcome the stiffness of the collision operator and capture the fluid dynamic limit, a class of asymptotic-preserving schemes \cite{JinYan11} were introduced. Furthermore, a stochastic Galerkin method was developed to deal with the Landau equation with uncertainties \cite{HuJinShu18}. 
However, how to solve equation \eqref{eqn: homogeneous Landau} efficiently still remains a major concern. The complexity of evaluating the quadratic collision operator $Q(f,f)$ is of $O(N^2)$ where $N$ is the number of discrete velocity points. Some fast algorithms are investigated for reducing the cost to $O(N\log N)$ including the multigrid algorithm \cite{BuetCordierDegondLemou97} and the one combined with fast multipole expansion \cite{Lemou98}. A Fourier spectral method is developed in \cite{PareschiRussoToscani00}, which is $O(N\log N)$ by using the fast Fourier transform thanks to the convolutional property. Recently, an Hermite spectral method \cite{LiWangWang20} is presented, where surrogate collision models are used to accelerate the computation. For the nonhomogeneous case, time splitting strategies are adopted, we refer readers to \cite{FilbetPareschi02,DimarcoLiPareschiYan15,ZhangGamba17,HuJinShu18,LiRenWang21}.

Particle methods have received profound development in the past several decades, see the review paper \cite{Chertock17} and the references therein. In the particle methods, the solution is approximated by the linear combination of Dirac delta-functions located at the particles. The particle locations and weights are evolved in time according to the ODE systems obtained from the weak formulation of the target equation. For diffusive-type equations, it seems that the existing particle methods do not keep the gradient flow structure of the equation except for the porous medium equation \cite{LionsGallic01}. In order to make sense of the entropy functional and maintain the gradient flow structure of the homogeneous Landau equation at the particle level, Carrillo et al. \cite{CarrilloHuWangWu20} provided two kinds of regularization of the entropy functional following \cite{CarrilloCraigPatacchini19}. Hence, the deterministic particle methods can preserve the basic properties of the Landau equation. Carrillo et al. \cite{CarrilloDelgadinoDesvillettesWu20} rigorously studied the gradient flow structure under a tailored metric inspired by the gradient flow structure to the Boltzmann equation \cite{Erbar}. The gradient flow structures of both equations are rigorously connected through the grazing collision limit via $\Gamma$-convergence of the gradient flows \cite{carrillo2021boltzmann}. 

Though having the good properties, the particle method in \cite{CarrilloHuWangWu20} is still $O(N^2)$. With the help of treecode summation, it can be reduced to $O(N\log N)$. Motivated by the random batch method (RBM) \cite{JinLiLiu20}, our objective in this work is to introduce a stochastic particle method for the homogeneous Landau equation with only $O(N)$ cost. RBM is devoted to simplify the pair-wise interactions between particles. Utilizing small but randomly selected--for each time step-- batch strategy, the interactions only occur inside the small batches so that the computational cost is reduced from $O(N^2)$ to $O(N)$ per time step. Using a random 
mini-batch was popular in machine learning, known as the stochastic gradient descent method \cite {SGD}. The random binary collision
idea was also proposed for mean-field equation of flocking and swarming \cite{AP2013, CPZ}. Due to the simplicity and scalability, it already has a variety of applications in solving mean-field Poisson-Nernst-Planck and Poisson-Boltzmann equations \cite{LiLiuTang21}, efficient sampling \cite{SVGDRBM,RBMC}, molecular dynamics \cite{RBE,LiangTanZhaoLiJinHongXu21} and so on, see \cite{JinLi21} for the recent review of RBM. 

In this paper, we propose a random batch particle method for the homogeneous Landau equation \eqref{eqn: homogeneous Landau}. It can be applied to both types of regularization of the entropy functional as introduced in \cite{CarrilloHuWangWu20}. The random batch idea is used to reduce the cost of binary collisions. In addition, taking advantage of the rapid decay property of the mollifier, we further
reduce the cost of approximation to the solution and the first variation of the regularized entropy. Hence, the final cost is $O(N)$ per time step. As we shall see, the random batch particle methods can retain the conserved quantities as well.

The rest of the paper is organized as follows. In section \ref{sec: Landau}, we briefly introduce two types of the regularized homogeneous Landau equation and the corresponding particle methods they induce. In section \ref{sec: RBM Landau}  the random batch particle methods are introduced and we also study their conservation and dissipating properties. Numerical examples are given in section \ref{sec: numer}, which validate the efficiency and accuracy of our random batch particle methods. The paper is concluded in section \ref{sec: conclusion}.

\section{The regularized homogeneous Landau equations and their particle approximations}
\label{sec: Landau}

\subsection{The regularized homogeneous Landau equations}

In this section, we give a brief review of the regularized homogeneous Landau equations in \cite{CarrilloHuWangWu20}. 
Rewrite equation \eqref{eqn: homogeneous Landau} as a nonlinear continuity equation
\begin{equation}
	\partial_t f = Q(f,f)=-\nabla_v \cdot (U(f)f),
	\label{eqn: Landau}
\end{equation}
with velocity field 
$$
U(f)=-\int_{\mathbb{R}^d} A(v-v_*)\left(\nabla_v \frac{\delta E}{\delta f} - \nabla_{v_*} \frac{\delta E_*}{\delta f_*}\right)f_*\mathrm{d}v_*.
$$
As mentioned in the introduction, one can easily obtain a formal gradient flow structure and generalize to the regularized equations using this form \cite{CarrilloDelgadinoDesvillettesWu20}.

Consider the mollifier
\begin{equation}
	\psi_{\epsilon}(v)=\frac{1}{(2\pi \epsilon)^{d/2}} \exp\left(-\frac{|v|^2}{2\epsilon}\right).
\label{eqn: mollifier}
\end{equation}
We can take into consideration two types of regularization. Define the regularized entropy as follows:
\begin{enumerate}[type I]
	\item  
\begin{equation}
	E_{\epsilon}(f)=\int_{\mathbb{R}^d} (f* \psi_{\epsilon})\log (f * \psi_{\epsilon}) \mathrm{d}v.
\label{eqn: regular entropy I}
\end{equation}

\item  
\begin{equation}
  \tilde{E}_{\epsilon}(f)=\int_{\mathbb{R}^d} f\log (f * \psi_{\epsilon}) \mathrm{d}v.
\label{eqn: regular entropy II}
\end{equation}
\end{enumerate}
By simple computation, one has \cite{CarrilloCraigPatacchini19}
$$
\frac{\delta E_{\epsilon}}{\delta f}
=\psi_{\eps}*(\log(f*\psi_{\epsilon})+1),
\quad \nabla\frac{\delta E_{\epsilon}}{\delta f}
=\nabla \psi_{\eps}*\log(f*\psi_{\epsilon}),$$
$$\frac{\delta \tilde{E}_{\epsilon}}{\delta f}=\log
(f*\psi_{\epsilon})+\left(\frac{f}{f*\psi_{\epsilon}}\right)*\psi_{\eps},
\quad \nabla\frac{\delta \tilde{E}_{\epsilon}}{\delta f}
=\frac{f*\nabla\psi_{\epsilon}}{f*\psi_{\epsilon}}+\left(\frac{f}{f*\psi_{\epsilon}}\right)*\nabla \psi_{\eps}.
$$

Then the regularized homogeneous Landau equation of type I is given by
\begin{equation}
	\partial_t f = Q_{\epsilon}(f,f)=-\nabla_v \cdot (U_{\epsilon}(f)f),
	\label{eqn: Landau regular I}
\end{equation}
where 
$$
U_{\epsilon}(f)=-\int_{\mathbb{R}^d} A(v-v_*)\left(\nabla_v \frac{\delta E_{\epsilon}}{\delta f} - \nabla_{v_*} \frac{\delta E_{\epsilon,*}}{\delta f_*}\right)f_*\mathrm{d}v_*.
$$
The regularized Landau equation \eqref{eqn: Landau regular I} satisfies the following properties \cite{CarrilloHuWangWu20}:
\begin{enumerate}
	\item Conservation of mass, momentum and energy:
	\begin{equation}
	\dfrac{\mathrm{d}}{\mathrm{d}t}\dint_{\mathbb{R}^d} 
	\begin{pmatrix}
		1\\
		v\\
		|v|^2
	\end{pmatrix}
	f \mathrm{d}v=\bm{0}.
	\label{eqn: conserved}
	\end{equation}
	\item Dissipation of entropy: let $B^{\epsilon}_{v,v_*}=\nabla_v \frac{\delta E_{\epsilon}}{\delta f} - \nabla_{v_*} \frac{\delta E_{\epsilon,*}}{\delta f_*}$, then
	\begin{equation}
	    \frac{\mathrm{d}E_{\epsilon}}{\mathrm{d}t}=-D_{\epsilon}=-\frac{1}{2}\iint_{\mathbb{R}^{2d}}B^{\epsilon}_{v,v_*}\cdot A(v-v_*)B^{\epsilon}_{v,v_*}ff_*\mathrm{d}v\mathrm{d}v_*\leq 0.
	    \label{eqn: dissipation}
	\end{equation}
	
	\item The stationary state of \eqref{eqn: Landau regular I} is also characterized by a Maxwellian.
\end{enumerate}

Similarly, the regularized homogeneous Landau equation of type II reads
\begin{equation}
	\partial_t f = \tilde{Q}_{\epsilon}(f,f)=-\nabla_v \cdot (\tilde{U}_{\epsilon}(f)f),
	\label{eqn: Landau regular II}
\end{equation}
with
$$
\tilde{U}_{\epsilon}(f)=-\int_{\mathbb{R}^d} A(v-v_*)\left(\nabla_v \frac{\delta \tilde{E}_{\epsilon}}{\delta f} - \nabla_{v_*} \frac{\delta \tilde{E}_{\epsilon,*}}{\delta f_*}\right)f_*\mathrm{d}v_*.
$$
Equation \eqref{eqn: Landau regular II} also shares the aforementioned conservation \eqref{eqn: conserved} and entropy dissipation \eqref{eqn: dissipation} with $E_{\epsilon}$ being replaced by $\tilde{E}_{\epsilon}$. 

\subsection{Deterministic particle methods}

In order to numerically approximate the Landau equation, Carrillo et al. \cite{CarrilloHuWangWu20} derived two deterministic particle methods based on the gradient flow structure of the regularized Landau equations \eqref{eqn: Landau regular I} and \eqref{eqn: Landau regular II}. The particle methods can preserve the basic conservation and entropy decay properties of the Landau operator $Q(f,f)$. Now, we briefly introduce these methods. 

\subsubsection{Type I method}

Approximate $f$ by the $N$ particle formulation 
\begin{equation}
 	f^N(t,v)=\sum_{i=1}^N w_i \delta (v-v_i(t)).
 	\label{eqn: empirical}
\end{equation} 
Here, $v_i(t)$ and $w_i$ are the velocity and weight of particle $i$, $N$ is the total number of particles. \eqref{eqn: empirical} is
the so-called empirical measure in distributional sense, this is why we need to regularize the entropy.
The corresponding discrete entropy of type I is
\begin{equation}
	\begin{aligned}
	E_{\epsilon}^N&=E_{\epsilon}(f^N)=\int_{\mathbb{R}^d} (f^N* \psi_{\epsilon})\log (f^N* \psi_{\epsilon}) \mathrm{d}v\\
	&=\int_{\mathbb{R}^d}\left(\sum_{i=1}^N w_i\psi_{\epsilon}(v-v_i)\right)\log\left(\sum_{k=1}^N w_k\psi_{\epsilon}(v-v_k)\right)\mathrm{d}v.
\end{aligned}
\label{eqn: entropy I}
\end{equation}
Plug \eqref{eqn: empirical} into \eqref{eqn: Landau regular I}, one can obtain the evolution for the particle velocity
\begin{equation}
	\frac{\mathrm{d}v_i(t)}{\mathrm{d}t}=U_{\epsilon}(f^N)(t,v_i(t))=-\sum_{j=1}^N w_jA(v_i-v_j)\left[\nabla \frac{\delta E_{\epsilon}^N}{\delta f}(v_i)- \nabla \frac{\delta E_{\epsilon}^N}{\delta f}(v_j)\right],
	\label{eqn: particle I}
\end{equation}
where 
\begin{equation}
	\nabla \frac{\delta E_{\epsilon}^N}{\delta f}(v)=\int_{\mathbb{R}^d}\nabla \psi_{\epsilon}(v-u)\log\left(\sum_{k=1}^N w_k\psi_{\epsilon}(u-v_k)\right)\mathrm{d}u.
\label{eqn: grad of variation I}
\end{equation}

Furthermore, apply the second order composite mid-point quadrature rule to approximate the velocity integral in \eqref{eqn: grad of variation I}, one can obtain a discrete-in-velocity particle method. Truncate the whole space by a velocity computational domain $\Omega=[-L,L]^d$ with $L$ large enough. Let $n_o$ be the number of particles per dimension, then the mesh size is $h=2L/n_o$, the total particle number is $N=n_o^d$. Denote the squares of mesh as $Q_l, l=1,\cdots, N$, the velocity grid points $v_l^c$ are the center of $Q_l$.	
	
Hence, the fully discretized regularized entropy is
\begin{equation}
\bar{E}_{\epsilon}^N=\sum_{l=1}^N h^d \left(\sum_{i=1}^N w_i\psi_{\epsilon}(v_l^c-v_i)\right)\log\left(\sum_{k=1}^N w_k\psi_{\epsilon}(v_l^c-v_k)\right),
\label{eqn: dis_entropy I}	
\end{equation}
the gradient to the fully discretized first variation of the entropy is
\begin{equation}
\bar{F}^N_{\epsilon}(v_i):=\nabla \frac{\delta \bar{E}_{\epsilon}^N}{\delta f}(v_i)=\sum_{l=1}^N h^d \nabla \psi_{\epsilon} (v_i-v_l^c)\log\left(\sum_{k=1}^N w_k\psi_{\epsilon}(v_l^c-v_k)\right).
\label{eqn: dis_variation I}	
\end{equation}
The discrete-in-velocity particle method reads
\begin{equation}
	\frac{\mathrm{d}v_i}{\mathrm{d}t}=\bar{U}_{\epsilon}(f^N)(t,v_i)
	=-\sum_{j=1}^N w_j A(v_i-v_j)\left[ \bar{F}^N_{\epsilon}(v_i)-\bar{F}^N_{\epsilon}(v_j) \right].
	\label{eqn: discrete particle I}
\end{equation}
The particle method \eqref{eqn: discrete particle I} is a deterministic particle method for the homogeneous Landau equation \eqref{eqn: Landau} which guarantees the conservation of discrete mass, momentum and energy exactly, and the discrete entropy \eqref{eqn: dis_entropy I} dissipation up to $O(h^2)$ \cite{CarrilloHuWangWu20}.

Using the forward Euler method for time discretization, we summarize the particle method \eqref{eqn: discrete particle I} in Alg. \ref{algo:I}. Note that for Alg. \ref{algo:I}, the discrete mass and momentum are conserved exactly, while the discrete energy is only conserved up to $O(\Delta t)$. This is stated in \cite{CarrilloHuWangWu20} and we will give a proof in the random batch setting in section \ref{subsec: RBM type I}.

\begin{algorithm}[H]
\caption{Particle method for the homogeneous Landau equation (type I)}
{\bf Input} The number of particles per dimension is $n_o$, then $N=n_o^d$. Start time $t_0$ and terminal time $t_{end}$, time step $\Delta t$. Regularizing parameter $\epsilon$. Velocity domain $[-L, L]^d$ with $L>0$, then $h=2L/n_o$. Square centers $v_l^c, l=1,\cdots, N$. 

{\bf Initialization} Particles $v_k^0$ and corresponding weights $w_k$, $k=1,\cdots,N$. 

At each time step $n=0,1,\cdots,\lfloor \frac{t_{end}-t_0}{\Delta t}\rfloor$:
\begin{algorithmic}[1]
\State Step 1: Compute the blob solution \cite{CarrilloCraigPatacchini19} for all the square centers (cost $O(N^2)$)
$$f_l^c=\sum_{k=1}^N w_k \psi_{\epsilon}(v_l^c-v_k^n), \quad \forall \ l=1,\cdots, N.$$ 
\State Step 2: Compute the gradient to fully discretized first variation of the regularized entropy
(cost $O(N^2)$)
$$F_i=\sum_{l=1}^N h^d \nabla \psi_{\epsilon} (v_i^n-v_l^c)\log f_l^c,\quad \forall \ i=1,\cdots, N. $$
\State Step 3: Update the particle velocities (cost $O(N^2)$)
$$\frac{v_i^{n+1}-v_i^n}{\Delta t}=-\sum_{j=1}^N w_j A(v_i^n-v_j^n)( F_i-F_j),\quad \forall \ i=1,\cdots, N.$$
\end{algorithmic}
{\bf Output} Particle velocities at $t_{end}$.
\label{algo:I}
\end{algorithm}

\subsubsection{Type II method}

The corresponding entropy of type II under the particle formulation \eqref{eqn: empirical} is
\begin{equation}
	\tilde{E}_{\epsilon}^N=\tilde{E}_{\epsilon}(f^N)=\int_{\mathbb{R}^d} f^N\log (f^N* \psi_{\epsilon}) \mathrm{d}v=\sum_{i=1}^N w_i\log\left(\sum_{k=1}^N w_k\psi_{\epsilon}(v_i-v_k)\right).
\label{eqn: dis_entropy II}
\end{equation}
Analogously to type I, the evolution for the particle velocity of type II is
\begin{equation}
	\frac{\mathrm{d}v_i}{\mathrm{d}t}=\tilde{U}_{\epsilon}(f^N)(t,v_i(t))=-\sum_{j=1}^N w_j A(v_i-v_j)\left[\nabla \frac{\delta \tilde{E}_{\epsilon}^N}{\delta f}(v_i)- \nabla \frac{\delta \tilde{E}_{\epsilon}^N}{\delta f}(v_j)\right],
	\label{eqn: particle II}
\end{equation}
where 
$$
\nabla \frac{\delta \tilde{E}_{\epsilon}^N}{\delta f}(v)=\frac{\sum_{k=1}^N w_k\nabla\psi_{\epsilon}(v-v_k)}{\sum_{k=1}^N w_k\psi_{\epsilon}(v-v_k)}+\sum_{k=1}^N w_k\frac{\nabla \psi_{\epsilon}(v-v_k)}{\sum_{m=1}^N w_m \psi_{\epsilon}(v_k-v_m)}.
$$
Since there is no convolution outside logarithmic term in the regularized entropy \eqref{eqn: dis_entropy II}, this type of regularization is free of numerical quadrature in velocity. The discrete mass, momentum and energy of \eqref{eqn: particle II} are conserved, while the discrete entropy \eqref{eqn: dis_entropy II} is dissipated exactly too.

Similarly, using the forward Euler method for time discretization of \eqref{eqn: particle II} yields Alg. \ref{algo:II}, whose discrete energy is also conserved up to $O(\Delta t)$ due to first order Euler in time.

\begin{algorithm}[H]
\caption{Particle method for the homogeneous Landau equation (type II)}
{\bf Input} The number of particles per dimension is $n_o$, then $N=n_o^d$. Start time $t_0$ and terminal time $t_{end}$, time step $\Delta t$. Regularizing parameter $\epsilon$. 

{\bf Initialization} Particles $v_i^0$ and corresponding weights $w_i$, $i=1,\cdots,N$. 

At each time step $n=0,1,\cdots,\lfloor \frac{t_{end}-t_0}{\Delta t}\rfloor$:
\begin{algorithmic}[1]
\State Step 1: Compute the regularized distribution (cost $O(N^2)$)
$$f_i=\sum_{k=1}^N w_k \psi_{\epsilon}(v_i^n-v_k^n), \quad \forall \ i=1,\cdots, N.$$ 
\State Step 2: Compute the gradient to the discretized first variation of the regularized entropy (cost $O(N^2)$)
$$F_i=\sum_{k=1}^N w_k\nabla\psi_{\epsilon}(v_i^n-v_k^n)\left(\frac{1}{f_i}+\frac{1}{f_k}\right),\quad \forall \ i=1,\cdots, N. $$
\State Step 3: Update the particle velocities (cost $O(N^2)$)
$$\frac{v_i^{n+1}-v_i^n}{\Delta t}=-\sum_{j=1}^N w_j A(v_i^n-v_j^n)( F_i-F_j) ,\quad \forall \ i=1,\cdots, N.$$
\end{algorithmic}
{\bf Output} Particle velocities at $t_{end}$.
\label{algo:II}
\end{algorithm}

\section{Random batch particle methods for the regularized homogeneous Landau equation}
\label{sec: RBM Landau}

In either the semi-discrete form \eqref{eqn: particle I} or the discrete-in-velocity form \eqref{eqn: discrete particle I} and \eqref{eqn: particle II} of the above particle methods, one needs to sum over all the particles $v_j$ to evolve particle $v_i$, which leads to an $O(N^2)$ computational cost for each time step. In order to significantly reduce the cost, we apply the random batch method \cite{JinLiLiu20} to this summation. At each time step, randomly divide the $N$ particles into $q=N/p$ batches $\mathcal{C}_v, v=1,\cdots, q$ with batch size $p\ll N$. Then particle $v_i$ will update itself only with those particles in the same batch.

\subsection{Type I RBM}
\label{subsec: RBM type I}

Mathematically, \eqref{eqn: particle I} is changed to
\begin{equation}
\begin{aligned}
	&\frac{\mathrm{d}v_i(t)}{\mathrm{d}t} = U^*_{\epsilon}(f^N)(t,v_i(t))\\
	&=-\frac{N-1}{p-1}\sum_{j\neq i, j\in \mathcal{C}_v} w_j A(v_i-v_j)\left[\nabla \frac{\delta E_{\epsilon}^N}{\delta f}(v_i)- \nabla \frac{\delta E_{\epsilon}^N}{\delta f}(v_j)\right], \quad i\in \mathcal{C}_v.
\end{aligned}
	\label{eqn: RBM particle I}
\end{equation}
Since the interactions only take place with the batch of $p$ particles, the computational cost is $O(pN)$ per time step.

Define the indicator function 
$$
I_{ij}=\left\{\begin{array}{ll}
1,\quad i,j \text{ in the same batch},\\
0,\quad i,j \text{ not in the same batch}.\\
\end{array}\right.
$$
Then, \eqref{eqn: RBM particle I} can be rewritten as
\begin{equation}
	\frac{\mathrm{d}v_i(t)}{\mathrm{d}t}=U^*_{\epsilon}(f^N)(t,v_i(t))
	=-\frac{N-1}{p-1}\sum_{j=1}^N w_jI_{ij} A(v_i-v_j)B_{ij}^N,
		\label{eqn: RBM semi}
\end{equation}
where
$B_{ij}^N=\nabla \frac{\delta E_{\epsilon}^N}{\delta f}(v_i)- \nabla \frac{\delta E_{\epsilon}^N}{\delta f}(v_j)$.
The next proposition shows that \eqref{eqn: RBM semi} inherits the desired properties due to the symmetry of $I_{ij}$. 

\begin{prop}
	The semi-discrete random batch particle method of type I \eqref{eqn: RBM semi} satisfies the following properties:
	\begin{enumerate}
		\item Conservation of mass, momentum and energy: 
		$$\frac{\mathrm{d}}{\mathrm{d}t}\sum_{i=1}^N w_i\phi(v_i)=0 \quad \text{for } \phi(v)=1, v, |v|^2.$$ 
		\item Dissipation of entropy: $\frac{\mathrm{d} E_{\epsilon}^N}{\mathrm{d}t}=-D_{\epsilon}^*\leq 0$, where
		$$
		D_{\epsilon}^*=\frac{N-1}{2(p-1)}\sum_{i=1}^N\sum_{j=1}^N w_iw_j I_{ij}B_{ij}^N\cdot A(v_i-v_j)B_{ij}^N.
		$$
	\end{enumerate}
	\label{prop:1}
\end{prop}
\begin{proof}
\begin{enumerate}
	\item Indeed, 
	$$
	\begin{aligned}
		\frac{\mathrm{d}}{\mathrm{d}t}\sum_{i=1}^N w_i\phi(v_i)
		&=~\sum_{i=1}^N w_i \nabla \phi(v_i) \cdot
		U^*_{\epsilon}(f^N)(t,v_i(t))\\
	&=-\frac{N-1}{p-1}\sum_{i=1}^N\sum_{j=1}^N w_i w_j I_{ij} \nabla\phi(v_i)\cdot A(v_i-v_j)B_{ij}^N\\
	&=-\frac{N-1}{2(p-1)}\sum_{i=1}^N\sum_{j=1}^N w_i w_j I_{ij}\left(\nabla\phi(v_i)-\nabla\phi(v_j)\right) \cdot A(v_i-v_j)B_{ij}^N,
	\end{aligned}
	$$
	which vanishes with $\phi(v)=1, v, |v|^2$ since $v \in \ker A(v)$.
\item It follows from \eqref{eqn: entropy I} that
\begin{align*}
	\frac{\mathrm{d}E_{\epsilon}^N}{\mathrm{d}t}=&\int_{\mathbb{R}^d}\sum_{i=1}^N w_i\nabla\psi_{\epsilon}(v-v_i)\cdot\left(-\frac{\mathrm{d} v_i(t)}{\mathrm{d}t}\right)\log\left(\sum_{k=1}^N w_k\psi_{\epsilon}(v-v_k)\right)\mathrm{d}v\\
& +\int_{\mathbb{R}^d}\left(\sum_{i=1}^N w_i\psi_{\epsilon}(v-v_i)\right)\frac{\sum_{k=1}^N w_k\nabla\psi_{\epsilon}(v-v_k) \cdot\left(-\frac{\mathrm{d} v_k(t)}{\mathrm{d}t}\right)}{\sum_{k=1}^N w_k\psi_{\epsilon}(v-v_k)}\mathrm{d}v\\
=&: I_1+I_2.
\end{align*}
Note that $$I_2=-\int_{\mathbb{R}^d}\sum\limits_{k=1}^N w_k\nabla\psi_{\epsilon}(v-v_k) \cdot\frac{\mathrm{d} v_k(t)}{\mathrm{d}t}\mathrm{d}v=\frac{\mathrm{d}}{\mathrm{d}t}\sum\limits_{k=1}^N w_k \int_{\mathbb{R}^d}\psi_{\epsilon}(v-v_k) \mathrm{d}v = 0 $$
since $\int_{\mathbb{R}^d}\psi_{\epsilon}(v-v_k) \mathrm{d}v=1$.
Besides,
$$
 \begin{aligned}
 	I_1&=\sum_{i=1}^N w_i
 	\left(\int_{\mathbb{R}^d}-\nabla\psi_{\epsilon}(v-v_i)\log\left(\sum_{k=1}^N w_k\psi_{\epsilon}(v-v_k)\right)\mathrm{d}v\right)\cdot\frac{\mathrm{d} v_i}{\mathrm{d}t}\\
 	&=\sum_{i=1}^N w_i \nabla \frac{\delta E_{\epsilon}^N}{\delta f}(v_i)\cdot\frac{\mathrm{d} v_i}{\mathrm{d}t}\\
 	&=-\frac{N-1}{p-1}\sum_{i=1}^N w_i \nabla \frac{\delta E_{\epsilon}^N}{\delta f}(v_i) \cdot \sum_{j=1}^N w_j I_{ij} A(v_i-v_j)B_{ij}^N\\
 	&=-\frac{N-1}{2(p-1)}\sum\limits_{i=1}^N\sum_{j=1}^N w_i w_j I_{ij}B_{ij}^N\cdot A(v_i-v_j)B_{ij}^N\\
 	&=-D_{\epsilon}^*\leq 0.
 \end{aligned}
$$
This completes the proof of entropy dissipation.
\end{enumerate}
\end{proof}

Similarly, the random batch version of the discrete-in-velocity particle method \eqref{eqn: discrete particle I} reads
\begin{equation}
	\frac{\mathrm{d}v_i}{\mathrm{d}t}=\bar{U}^*_{\epsilon}(f^N)(t,v_i)
	=-\frac{N-1}{p-1}\sum_{j=1}^N w_j I_{ij} A(v_i-v_j)\left[ \bar{F}^N_{\epsilon}(v_i)-\bar{F}^N_{\epsilon}(v_j) \right].
	\label{eqn: RBM discrete I}
\end{equation}

\begin{prop}
\eqref{eqn: RBM discrete I} also inherits the following properties:
\begin{enumerate}
	\item Conservation of mass, momentum and energy:
	$$
	\frac{\mathrm{d}}{\mathrm{d}t}\sum_{i=1}^N w_i
	\begin{pmatrix}
		1\\
		v_i\\
		|v_i|^2
	\end{pmatrix}
	=\bm{0}.
	$$
	\item Dissipation of entropy up to $O(h^2)$: 
	$$
	\bar{E}_{\epsilon}^N(t)-\bar{E}_{\epsilon}^N(0)=-\int_0^t \bar{D}_{\epsilon}^* \mathrm{d}s+ O(h^2),
	$$
	where 
	$$\bar{D}_{\epsilon}^*=\frac{N-1}{2(p-1)}\sum\limits_{i=1}^N\sum\limits_{j=1}^N w_i w_j I_{ij}\left(\bar{F}^N_{\epsilon}(v_i)-\bar{F}^N_{\epsilon}(v_j)\right)\cdot A(v_i-v_j)\left(\bar{F}^N_{\epsilon}(v_i)-\bar{F}^N_{\epsilon}(v_j) \right)\geq 0.$$
\end{enumerate}
\label{prop:2}
\end{prop}
The proof of conservation is the same as that in Proposition \ref{prop:1}. The proof of entropy decay is the discrete version of that in Proposition \ref{prop:1}, while the $O(h^2)$ error is brought by the mid-point composite quadrature rule.

Based on these, we apply RBM to the summation of Step 3 in Alg. \ref{algo:I}. Besides, we make use of the rapid decay property of the mollifier $\psi_{\epsilon}$ to reduce the summation in Steps 1 and 2. Denote $\sigma$ as the velocity distance to be considered in the summation of Steps 1 and 2, the computational cost is again $O(N)$ per time step resorting to the cell list data structure {\cite[Appendix F]{FrenkelSmit01}}. This gives rise to Alg. \ref{algo:RBM I} as follows.

\begin{algorithm}[H]
\caption{Random batch particle method for the homogeneous Landau equation (type I)}
{\bf Input} The number of particles per dimension is $n_o$, then $N=n_o^d$. Start time $t_0$ and terminal time $t_{end}$, time step $\Delta t$. Regularizing parameter $\epsilon$. Closeness parameter $\sigma$. Batch size $p\ll N$. Velocity domain $[-L, L]^d$ with $L>0$, then $h=2L/n_o$. Square centers $v_l^c, l=1,\cdots, N$. 

{\bf Initialization} Particles $v_k^0$ and corresponding weights $w_k$, $k=1,\cdots,N$. 

At each time step $n=0,1,\cdots,\lfloor \frac{t_{end}-t_0}{\Delta t}\rfloor$, do the following:
\begin{algorithmic}[1]
\State Step 1: Compute the blob solution \cite{CarrilloCraigPatacchini19} using only the particles within velocity distance $\sigma$ to $v_l^c$ (cost $O(N)$)
$$f_l^c=\sum_{k\in \mathcal{A}_l} w_k \psi_{\epsilon}(v_l^c-v_k^n), \quad \mathcal{A}_l=\{k: \abs{v_l^c-v_k^n}\leq \sigma \}, \quad \forall \ l=1,\cdots, N.$$ 
\State Step 2: Compute the gradient to the fully discretized first variation of entropy using only the square centers within velocity distance $\sigma$ to $v_i^n$ (cost $O(N)$)
$$F_i=\sum_{l\in \mathcal{B}_i} h^d \nabla \psi_{\epsilon} (v_i^n-v_l^c)\log f_l^c, \quad \mathcal{B}_i=\{l: \abs{v_i^n-v_l^c}\leq \sigma \},
\quad \forall \ i=1,\cdots, N. $$
\State Step 3: Divide $N$ particles into $q=N/p$ batches $\mathcal{C}_v$ randomly.

For each batch $\mathcal{C}_v, v=1,\cdots, q$, update the particles by 
\begin{equation}
	\frac{v_i^{n+1}-v_i^n}{\Delta t}=-\frac{N-1}{p-1}\sum_{j\in \mathcal{C}_v } w_j A(v_i^n-v_j^n)( F_i-F_j) ,\quad \forall \ i\in \mathcal{C}_v.
	\label{eq: RBM}
\end{equation}
(The interactions only take place among the particles in the same batch, thus the computational cost is $O(Np)$.)
\end{algorithmic}
{\bf Output} Particle velocities at $t_{end}$.
\label{algo:RBM I}
\end{algorithm}

Likewise, due to the symmetry of $I_{ij}$, one has
\begin{prop}
The discrete mass and momentum are conserved, while the discrete energy is conserved up to $O(\Delta t)$ for Alg. \ref{algo:RBM I}.
Namely,
\begin{equation*}
   \sum_i w_i=const,\quad \frac{\sum_i w_i v_i^{n+1}- \sum_i w_i v_i^{n}}{\Delta t}=0,\quad \frac{\sum_i w_i |v_i^{n+1}|^2- \sum_i w_i |v_i^{n}|^2}{\Delta t}=O(\Delta t).
\end{equation*}
\label{prop:3}
\end{prop}
\begin{proof}
The mass conservation is automatical. To see the momentum conservation, multiplying \eqref{eq: RBM} by $w_i$ and sum over $i$, one gets
$$
\frac{\sum_i w_i v_i^{n+1}- \sum_i w_i v_i^{n}}{\Delta t}=-\frac{N-1}{p-1}\sum_{i=1}^N\sum_{j=1}^N w_i w_j I_{ij} A(v_i^n-v_j^n)( F_i-F_j)=0.
$$ 
For the discrete energy, denote the RHS of \eqref{eq: RBM} as $\bar{U}^*_{\epsilon}(f^N)(t,v_i^n)$, then we can obtain

$$
\begin{aligned}
	&\frac{\sum_i w_i |v_i^{n+1}|^2- \sum_i w_i |v_i^{n}|^2}{\Delta t}=\frac{\sum_i w_i \left|v_i^n +\Delta t \bar{U}^*_{\epsilon}(f^N)(t,v_i^n)\right|^2- \sum_i w_i \left|v_i^{n}\right|^2}{\Delta t}\\
=&~2\sum_{i=1}^N w_i v_i^n\bar{U}^*_{\epsilon}(f^N)(t,v_i^n)+\Delta t \sum_{i=1}^N w_i \left|\bar{U}^*_{\epsilon}(f^N)(t,v_i^n)\right|^2\\
=&~-\frac{2(N-1)}{p-1}\sum_{i=1}^N\sum_{j=1}^N w_i w_j I_{ij} A(v_i^n-v_j^n)( F_i-F_j)\cdot v_i^n +O(\Delta t)\\
=&~-\frac{(N-1)}{p-1}\sum_{i=1}^N\sum_{j=1}^N w_i w_j I_{ij} A(v_i^n-v_j^n)( F_i-F_j)\cdot (v_i^n-v_j^n)+O(\Delta t)\\
=&~O(\Delta t).
\end{aligned}
$$
\end{proof}

\subsection{Type II RBM}

Like in type I, in order to update the particle velocity $v_i$, one needs to sum over all the particles $v_j$. This is very time-consuming. So we can adopt the random batch idea to reduce the cost of \eqref{eqn: particle II} from $O(N^2)$ to $O(N)$ per time step.
That is, 
\begin{equation}
\begin{aligned}
	\frac{\mathrm{d}v_i}{\mathrm{d}t}=&~\tilde{U}_{\epsilon}^*(f^N)(t,v_i)\\
	=&-\frac{N-1}{p-1}\sum_{j\neq i, j\in \mathcal{C}_v} w_j A(v_i-v_j)\left[\nabla \frac{\delta \tilde{E}_{\epsilon}^N}{\delta f}(v_i)- \nabla \frac{\delta \tilde{E}_{\epsilon}^N}{\delta f}(v_j)\right],\quad i\in\mathcal{C}_v \\
=&-\frac{N-1}{p-1}\sum_{j=1}^N w_j I_{ij} A(v_i-v_j)\tilde{B}_{ij}^N,
\end{aligned}
	\label{eqn: RBM particle II}
\end{equation}
with $\tilde{B}_{ij}^N=\nabla \frac{\delta \tilde{E}_{\epsilon}^N}{\delta f}(v_i)- \nabla \frac{\delta \tilde{E}_{\epsilon}^N}{\delta f}(v_j)$. The random batch particle method of regularization type II \eqref{eqn: RBM particle II} inherits the desired properties due to the symmetry of $I_{ij}$.
\begin{prop}
	\eqref{eqn: RBM particle II} satisfies the following properties:
	\begin{enumerate}
		\item Conservation of mass, momentum and energy: 
		$$\frac{\mathrm{d}}{\mathrm{d}t}\sum_{i=1}^N w_i\phi(v_i)=0 \quad \text{for } \phi(v)=1, v, |v|^2.$$
		\item Dissipation of entropy: $\frac{\mathrm{d} \tilde{E}_{\epsilon}^N}{\mathrm{d}t}=-\tilde{D}_{\epsilon}^*\leq 0$, where
		$$
		\tilde{D}_{\epsilon}^*=\frac{N-1}{2(p-1)}\sum_{i=1}^N\sum_{j=1}^N w_iw_j I_{ij}\tilde{B}_{ij}^N\cdot A(v_i-v_j)\tilde{B}_{ij}^N.
		$$
	\end{enumerate}
	\label{prop:4}
\end{prop}
The proof is almost identical to that of Proposition \ref{prop:1}, so it is omitted.

Again, the random batch particle method of type II can be summarized to a algorithm whose overall computational cost is $O(N)$ per time step. Proposition \ref{prop:3} holds for Alg. \ref{algo:RBM II} as well.

\begin{algorithm}[H]
\caption{Random batch particle method for the homogeneous Landau equation (type II)}
{\bf Input} The number of particles per dimension is $n_o$, then $N=n_o^d$. Start time $t_0$ and terminal time $t_{end}$, time step $\Delta t$. Regularizing parameter $\epsilon$. Closeness parameter $\sigma$. Batch size $p\ll N$. 

{\bf Initialization} Particles $v_i^0$ and corresponding weights $w_i$, $i=1,\cdots,N$. 

At each time step $n=0,1,\cdots,\lfloor \frac{t_{end}-t_0}{\Delta t}\rfloor$, do the following:
\begin{algorithmic}[1]
\State Step 1: Compute the regularized distribution using only the particles within velocity distance $\sigma$ to $v_i^n$ (cost $O(N)$)
$$f_i=\sum_{k\in \mathcal{A}_i} w_k \psi_{\epsilon}(v_i^n-v_k^n), \quad \mathcal{A}_i=\{k: \abs{v_i^n-v_k^n}\leq \sigma \}, \quad \forall \ i=1,\cdots, N.$$ 

\State Step 2: Compute the gradient to discretized first variation of the entropy (cost $O(N)$)
$$F_i=\sum_{k\in \mathcal{A}_i} w_k\nabla\psi_{\epsilon}(v_i^n-v_k^n)
\left(\frac{1}{f_i}+\frac{1}{f_k}\right),\quad \forall \ i=1,\cdots, N. $$
\State Step 3: Divide $N$ particles into $q=N/p$ batches $\mathcal{C}_v$ randomly.

For each batch $\mathcal{C}_v, v=1,\cdots, q$, update the particles by 
\begin{equation}
	\frac{v_i^{n+1}-v_i^n}{\Delta t}=-\frac{N-1}{p-1}\sum_{j\in \mathcal{C}_v } w_j A(v_i^n-v_j^n)( F_i-F_j) ,\quad \forall \ i\in \mathcal{C}_v.
	\label{eq: RBM II}
\end{equation}
(The interactions only take place among the particles in the same batch, thus the computational cost is $O(Np)$.)
\end{algorithmic}
{\bf Output} Particle velocities at $t_{end}$.
\label{algo:RBM II}
\end{algorithm}

\section{Numerical experiments}
\label{sec: numer}

In this section, we compare the particle methods in \cite{CarrilloHuWangWu20} and our random batch particle methods for the two types of regularization for the homogeneous Landau equation. 
We give four classical numerical examples. It shows that our random batch particle methods have almost second order accuracy as the deterministic particle methods while the cost can be reduced remarkably to $O(N)$. 

\begin{exmp}[2D BKW solution for Maxwell molecules]
Consider the case when $d=2$, $\gamma=0$, the collision kernel is $A(z)=\frac{1}{16}(|z|^2I_d-z\otimes z)$. 
Requiring the macroscopic quantities to be $\rho=1$, $u=\bm{0}$, $T=1$(thus $\int_{\mathbb{R}^d} f |v|^2
\mathrm{d}v=d=2$), the exact solution of the homogeneous Landau equation is given by
$$
f^{ext}(t,v)=\frac{1}{2\pi K}\exp\left(-\frac{|v|^2}{2K}\right)\left(\frac{2K-1}{K}+\frac{1-K}{2K^2}|v|^2\right),\quad K=1-\frac{\exp(-t/8)}{2}.
$$
\label{eg:1}
\end{exmp}

Let $t_0=0, t_{end}=5$. The real distributions in $[-4,4]^2$ are in Fig. \ref{fig:1}.
\begin{figure}[H]
\centering
\includegraphics[width=8.5cm,height=3.5cm]{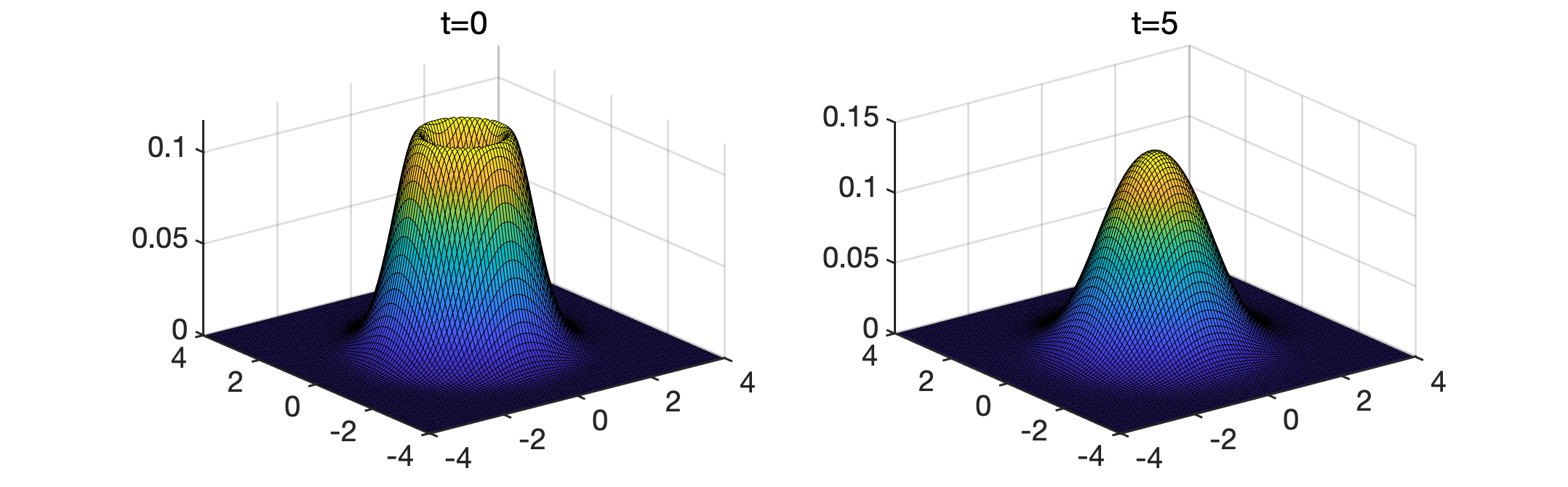}
\caption{Left: initial distribution. Right: exact distribution at $t_{end}$.}
\label{fig:1}
\end{figure}

First of all, we test the convergence and the computational cost with respect to particle numbers. The number of particles per dimension is chosen as 
$n_o=40, 60, 80, 100, 120, 140,$\\$ 160, 180, 200, 220, 240$ respectively, then the total particle number is $N=n_o^2$. Set $L=8$ for Algs. \ref{algo:I} and \ref{algo:RBM I}, $L=10$ for Algs. \ref{algo:II} and \ref{algo:RBM II}. The truncated length $L$ is chosen such that the particles do not escape from the computational velocity domain $\Omega=[-L, L]^2$ during their time span. We initialize the particle velocities uniformly in the support $[-4, 4]^2$, the weights are given according to $f^{ext}(0,v)$. Let $\Delta t = 0.01, h = 2L/n_o$. The default regularization parameter is $\epsilon=0.64h^{1.98}$ as in \cite{CarrilloHuWangWu20}. Besides, in the random batch particle methods: Algs. \ref{algo:RBM I} and \ref{algo:RBM II}, the closeness parameter $\sigma=4\sqrt{\epsilon}$, the number of batches per dimension $q_o=n_o/p_o$ is fixed as $5$. The error is measured by the relative $L_2$ error defined as 
$$\frac{\sqrt{\sum_{l=1}^N h^d\left|f^{ext}(t,v_l^c)-\psi_{\epsilon}*f^N(t,v_l^c)\right|^2}}{\sqrt{\sum_{l=1}^N h^d\left|f^{ext}(t,v_l^c)\right|^2}}.$$
It is clear from Fig. \ref{fig:new conv}(left) that all the four algorithms have second order decay in the relative $L_2$ error in $1/n_o$. 
As for the computational cost, we can conclude from Fig. \ref{fig:new conv}(right) that the cost for the original particle methods (Algs. \ref{algo:I}-\ref{algo:II}) are a bit lower than $O(N^2)$, while the cost for the random batch particle methods (Algs. \ref{algo:RBM I}-\ref{algo:RBM II}) are $O(N)$. 

\begin{figure}[H]
	\centering
	\includegraphics[width=5.5cm,height=5cm]{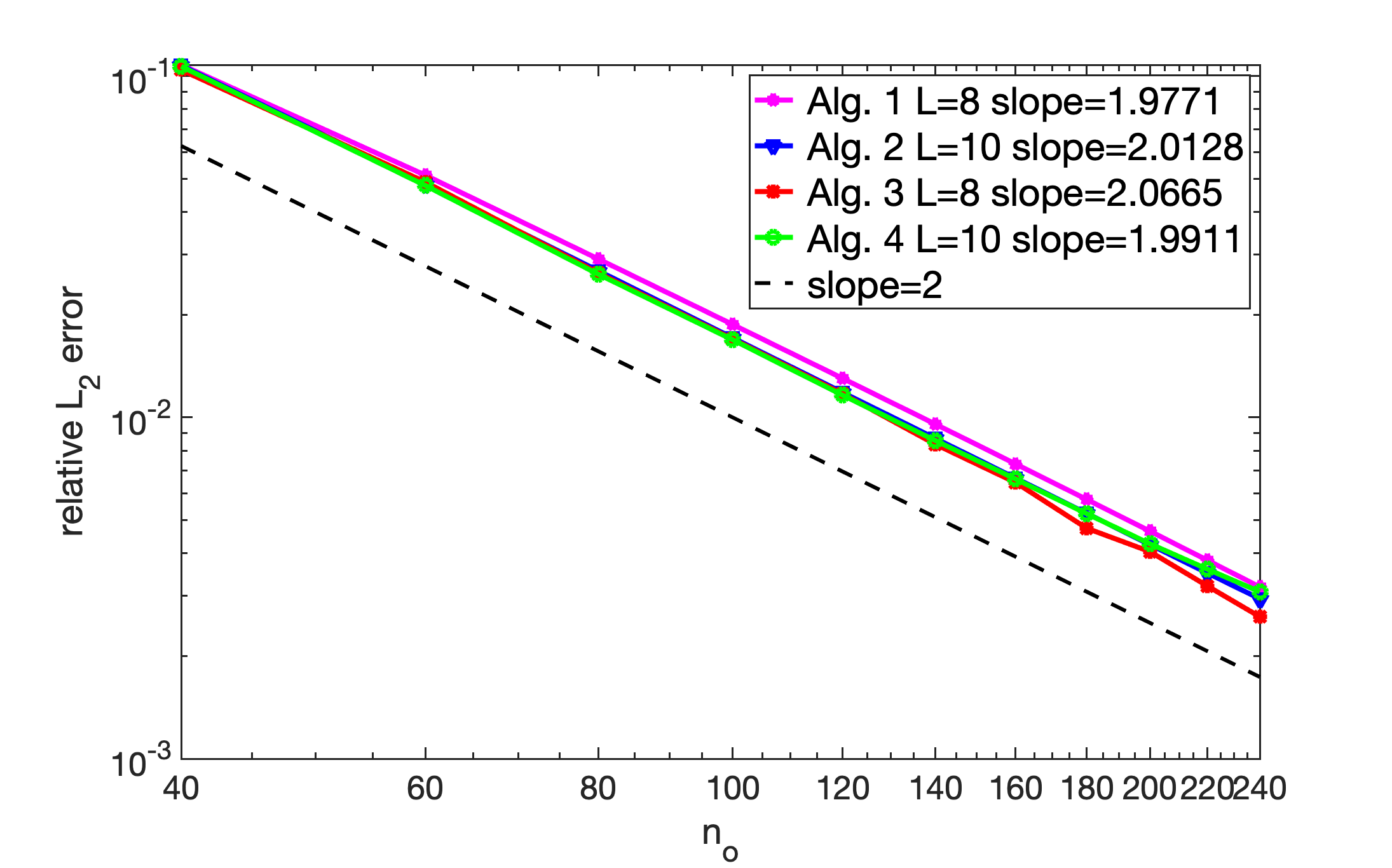}
	\includegraphics[width=5.5cm,height=5cm]{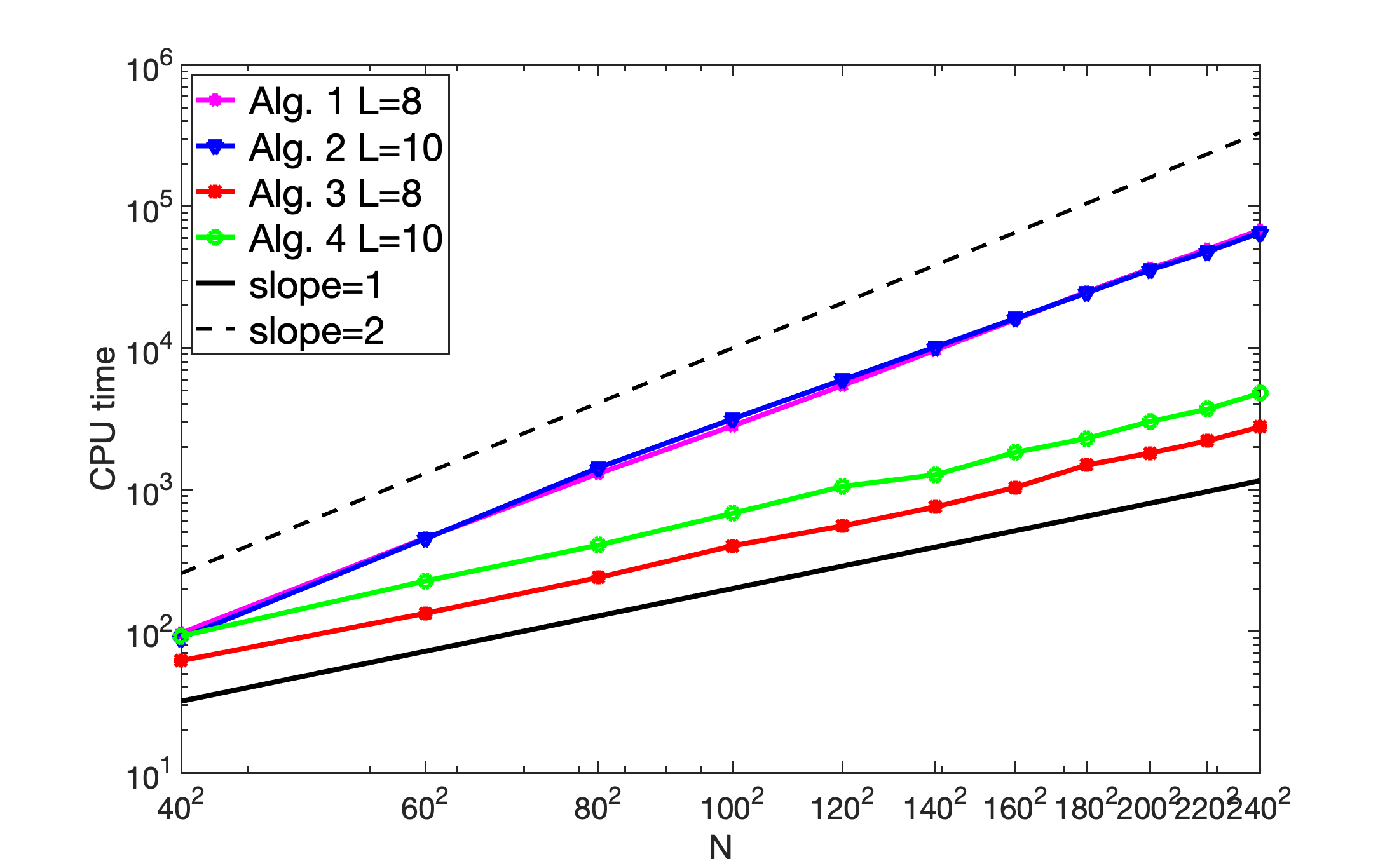}
\caption{Convergence result for Algs. \ref{algo:I} and \ref{algo:RBM I} when $L=8$ and Algs. \ref{algo:II} and \ref{algo:RBM II} when $L=10$.}
\label{fig:new conv}
\end{figure}

Next, we show the time evolution of the four algorithms in Fig. \ref{fig:n120}. It is clear that the total energy is conserved up to a small error, the momentum is conserved within machine precision, the entropy is dissipated, the relative $L_2$ error is stable. In Fig. \ref{fig:time2}, we plot the time evolution of Alg. \ref{algo:RBM I} with regard to different particle numbers. We can observe similar performance for the other algorithms. 
\begin{figure}[H]
	\centering
	\includegraphics[width=14cm,height=4cm]{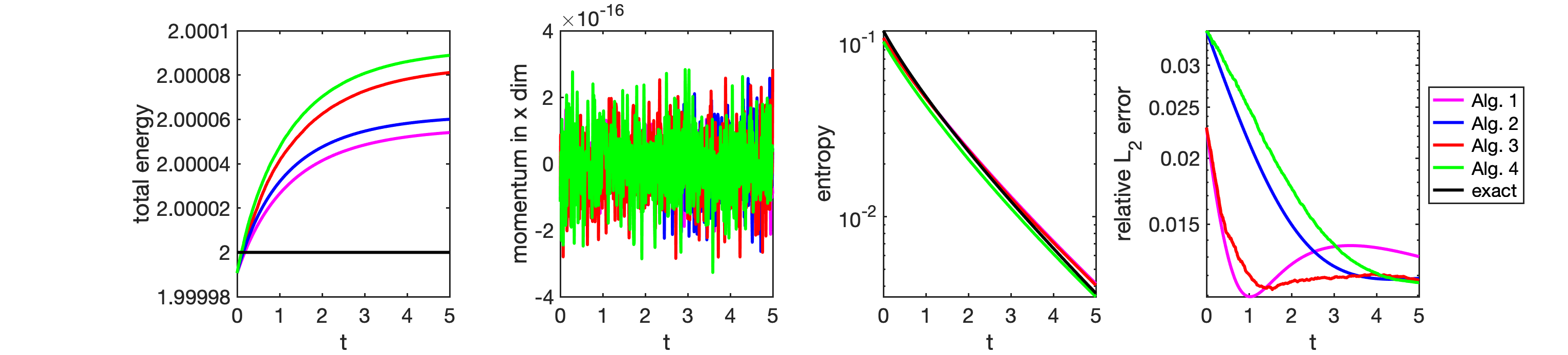}
\caption{Time evolution of the total energy, momentum in the $x$-dim, entropy, the relative $L_2$ error when $n_o=120$.}
\label{fig:n120}
\end{figure}

\begin{figure}[H]
	\centering
	\includegraphics[width=14cm,height=4cm]{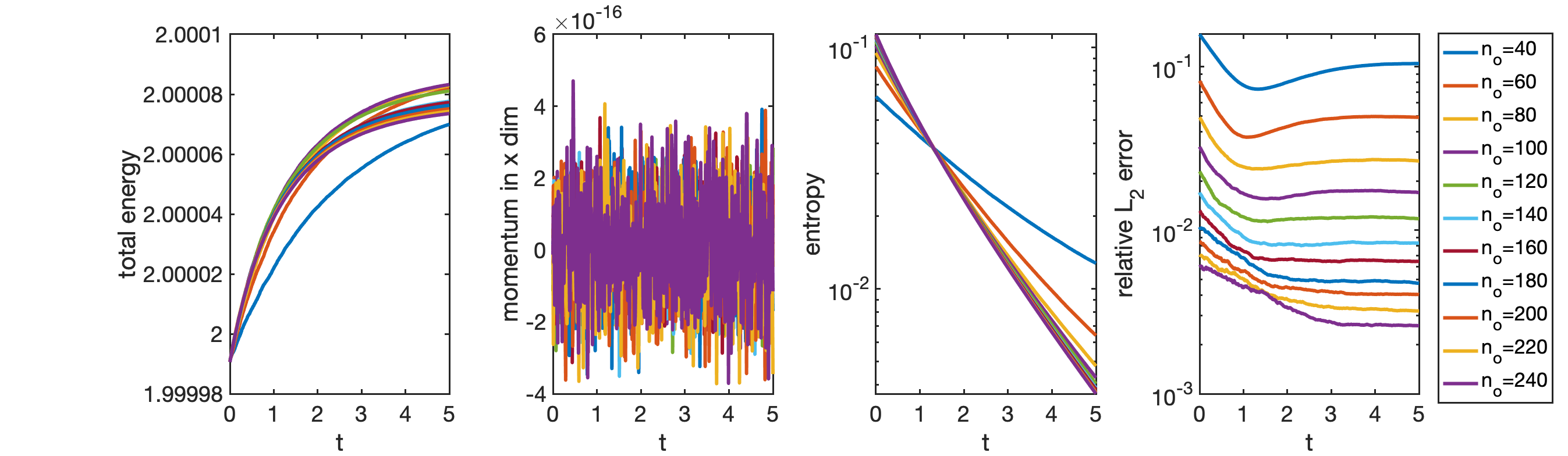}
\caption{Time evolution of the total energy, momentum in the $x$-dim, entropy, the relative $L_2$ error for Alg. \ref{algo:RBM I}.}
\label{fig:time2}
\end{figure}

\begin{exmp}[3D BKW solution for Maxwell molecules]
Consider the case when $d=3$, $\gamma=0$, the collision kernel is
$A(z)=\frac{1}{24}(|z|^2I_d-z\otimes z)$. Requiring the macroscopic quantities to be $\rho=1$, $u=\bm{0}$, $T=1$(thus $\int_{\mathbb{R}^d} f |v|^2 \mathrm{d}v=d=3$), the exact solution of the homogeneous Landau equation is given by
$$
f^{ext}(t,v)=\frac{1}{(2\pi K)^{3/2}}\exp\left(-\frac{|v|^2}{2K}\right)\left(\frac{5K-3}{2K}+\frac{1-K}{2K^2}|v|^2\right),\quad K=1-\exp(-t/6).
$$
\label{eg:2}
\end{exmp}

Let $t_0=5.5, t_{end}=6$. 
As in Example \ref{eg:1}, we test the convergence and the computational cost. The number of particles per dimension are chosen as $n_o=30, 40, 50, 60, 70, 80, 90$ respectively, then the total particle number is $N=n_o^3$. 
To ensure that the particles will always stay in $\Omega$, take $L=8$ and initialize the particles uniformly in $[-4,4]^3$. 
Take $\Delta t = 0.01, h = 2L/n_o, \epsilon=0.64h^{1.98}$, in Algs. \ref{algo:RBM I}-\ref{algo:RBM II}, let $\sigma=4\sqrt{\epsilon}$, $q_o=n_o/p_o=2$. 
The results are given in Fig. \ref{fig:new_conv_eg2}.
Seen from Fig. \ref{fig:new_conv_eg2}(left), one can observe roughly second order decay rate in $1/n_o$. The orders are lower for regularization type II. From Fig. \ref{fig:new_conv_eg2}(right), we see the cost for the original particle methods (megenta and blue solid lines) are a bit lower than $O(N^2)$. But the cost for the random batch methods (red and green solid lines) are higher than $O(N)$. For larger $n_o$, the slopes for random batch methods tend to $O(N^2)$. This is due to small value of $q_o$. We only use $q_o=2$ batches per dimension, so the saving by random batch is not very efficient while the extra cost brought by the cell list is heavy. So we also test the case where $q_o=5$, which means we use $1/125$ of total particles for different $N$. In this case, the speed up is obvious. The cost for random batch methods with $q_o=5$ is $O(N)$, see the red and green dashed lines in Fig. \ref{fig:new_conv_eg2}(right). However, one would sacrifice somewhat from accuracy for large $q_o$, see the red and green dashed lines in Fig. \ref{fig:new_conv_eg2}(left).
In addition, it is preferable to choose Alg. \ref{algo:RBM I} out of four in terms of both time and accuracy.

\begin{figure}[H]
\centering
\includegraphics[width=12.5cm,height=5cm]{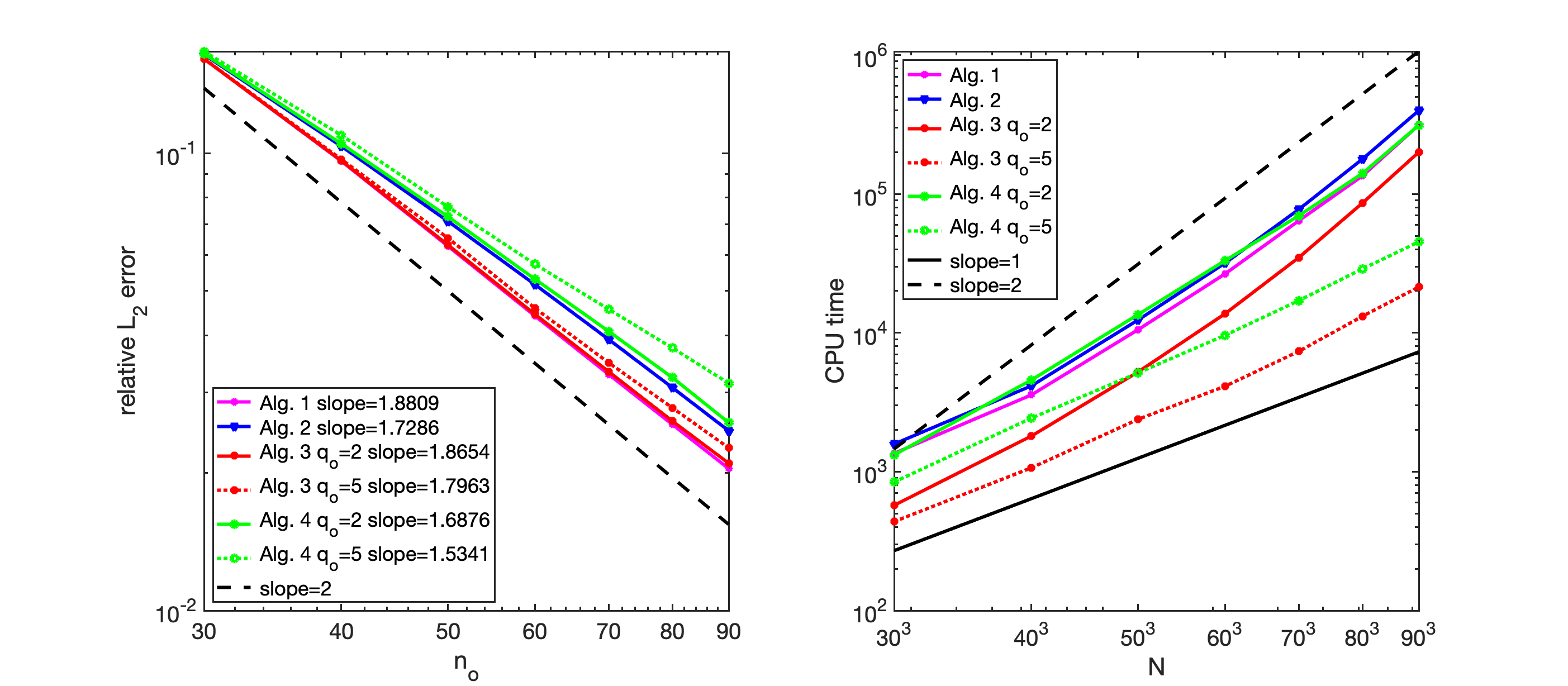}
\caption{Convergence results with $L=8$. Left: Relative $L^2$ norm of the error at $t_{end}$ w.r.t. different $n_o$. Right: CPU time w.r.t. different $N$.}
\label{fig:new_conv_eg2}
\end{figure}

Next, we show the performance of the four algorithms. 
Fig. \ref{fig:evolution} shows the evolution of conserved quantities and the relative $L_2$ error. The behavior is similar to that in Example \ref{eg:1}.
\begin{figure}[H]
\centering
\includegraphics[width=14cm,height=4cm]{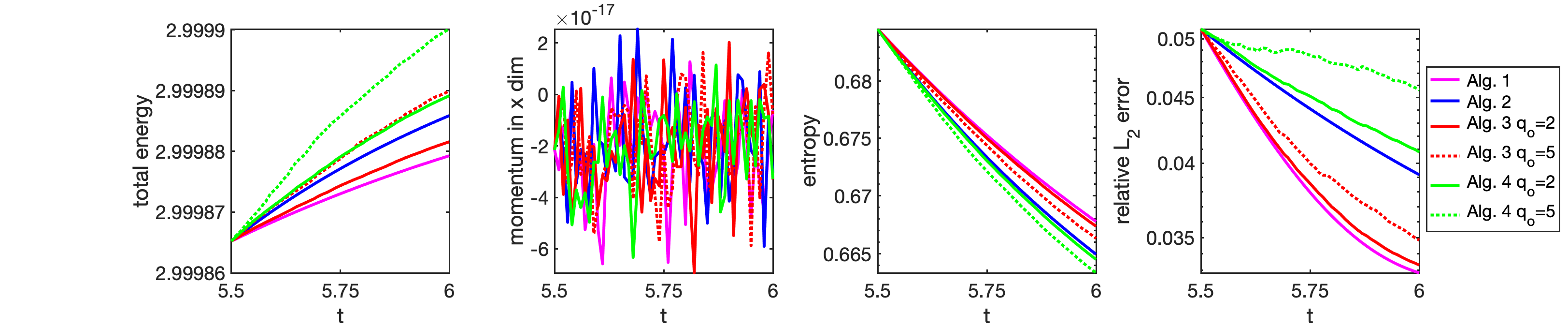}
\caption{Time evolution of the total energy, momentum in the $x$-dim, entropy, the relative $L_2$ error when $n_o=70$.}
\label{fig:evolution}
\end{figure}

Fig. \ref{fig:slice} illustrates the snapshots of the solutions $f(:, \frac{n_o}{2}, \frac{n_o}{2})$ at $t=5.5, 5.75, 6$ respectively when $n_o=70$. We also plot snapshots of the solutions versus different $n_o$ at $t=5.5, 5.75, 6$ in Fig. \ref{fig:slice_n} computed by Alg. \ref{algo:RBM I} with $q_o=5$, where we can observe better match as $n_o$ increases. 
\begin{figure}[H]
\centering
\includegraphics[width=14cm,height=4cm]{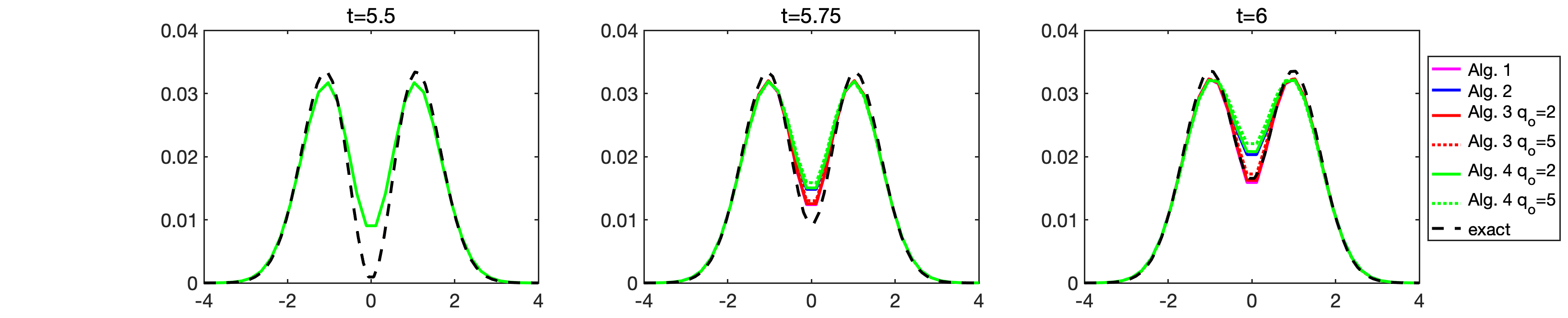}
\caption{Comparison of the four particle methods. Snapshots $f(:,\frac{n_o}{2},\frac{n_o}{2})$ 
at different times when $n_o=70$.}
\label{fig:slice}
\end{figure}

\begin{figure}[H]
\centering
\includegraphics[width=14cm,height=4cm]{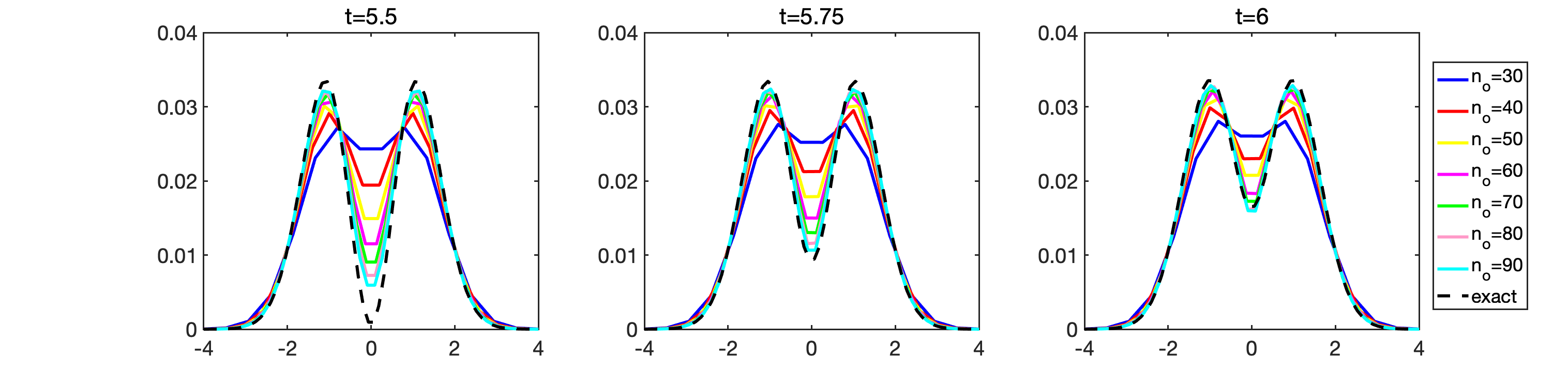}
\caption{Snapshots $f(:, \frac{n_o}{2}, \frac{n_o}{2})$ of Alg. \ref{algo:RBM I} with $q_o=5$ w.r.t different $n_o$ at different times.}
\label{fig:slice_n}
\end{figure}

\begin{exmp}[2D anistropic solution with Coulomb potential]
Consider the case when $d=2, \gamma=-3$, the collision kernel is $A(z)=\frac{1}{16}\frac{1}{|z|^3}(|z|^2I_d-z\otimes z)$. 
The initial condition is chosen to be bi-Maxwellian
$$
f(0,v)=\frac{1}{4\pi}\left\{\exp\left(-\frac{(v-u_1)^2}{2}\right)+\exp\left(-\frac{(v-u_2)^2}{2}\right)\right\},\quad u_1=(-2, 1), u_2=(0,-1).
$$
\label{eg:3}
\end{exmp}

For this example, we do not have the exact solution. Therefore, we only compare the performance of the four algorithms. Let $\Delta t=0.1, L=10, q_o=5$. The CPU time of the four algorithms in Fig. \ref{fig:time_eg3} coincides with the computational cost. Fig. \ref{fig:slices_t_eg3} illustrates the snapshots of the solutions $f(\frac{n_o}{2},:)$ at $t=0, 20, 40$ respectively when $n_o=180$. Finally, we use Alg. \ref{algo:I} with $n_o=180$ as a reference solution. For Alg. \ref{algo:RBM I}, test $n_o=40, 80, 120, 160$ respectively. We can observe better match as $n_o$ increases.

\begin{figure}[H]
\centering
\includegraphics[width=4.5cm,height=4cm]{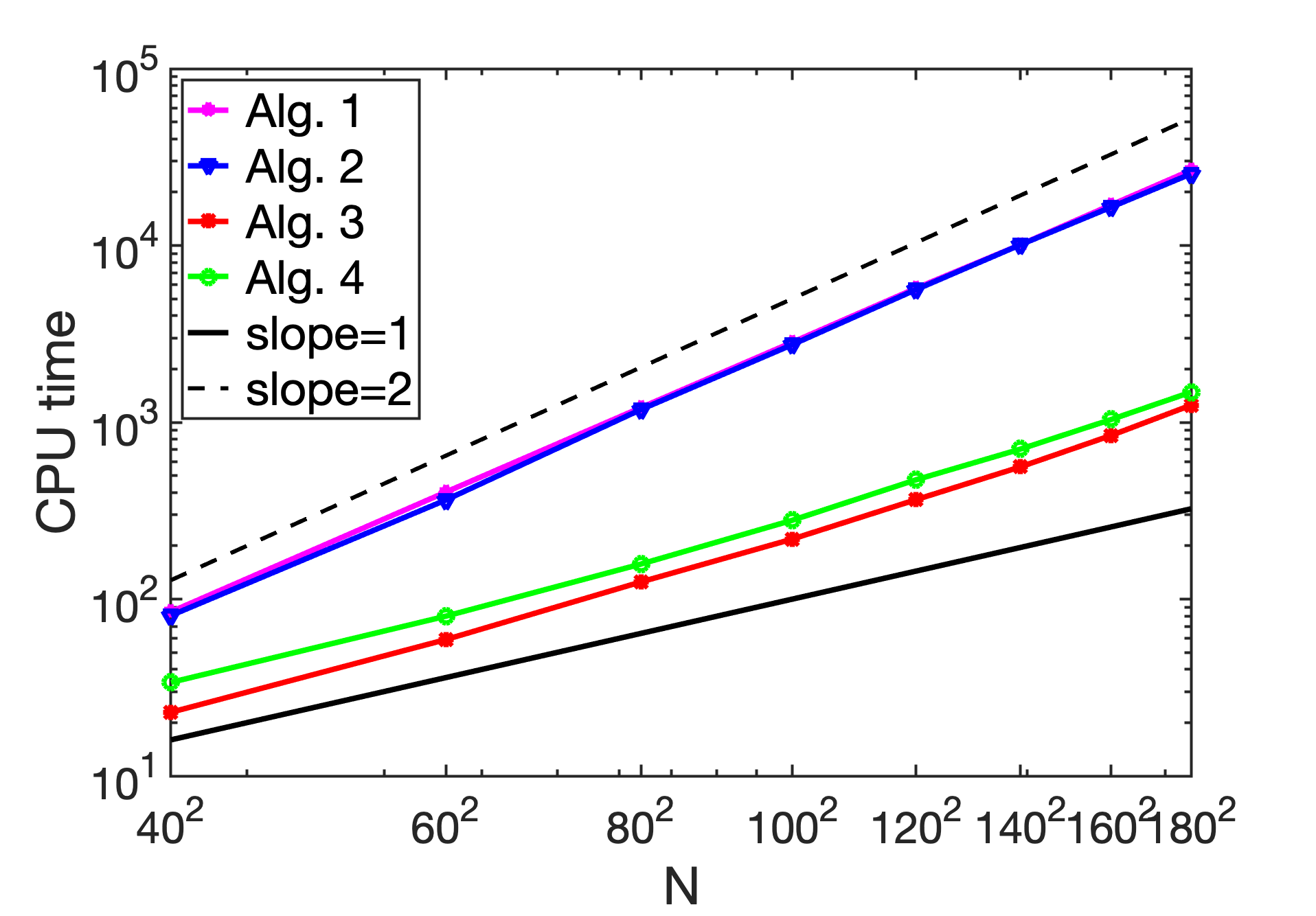}
\caption{CPU time of the four algorithms w.r.t different $N$ at time $t=20$.}
\label{fig:time_eg3}
\end{figure}

\begin{figure}[H]
\centering
\includegraphics[width=14cm,height=4cm]{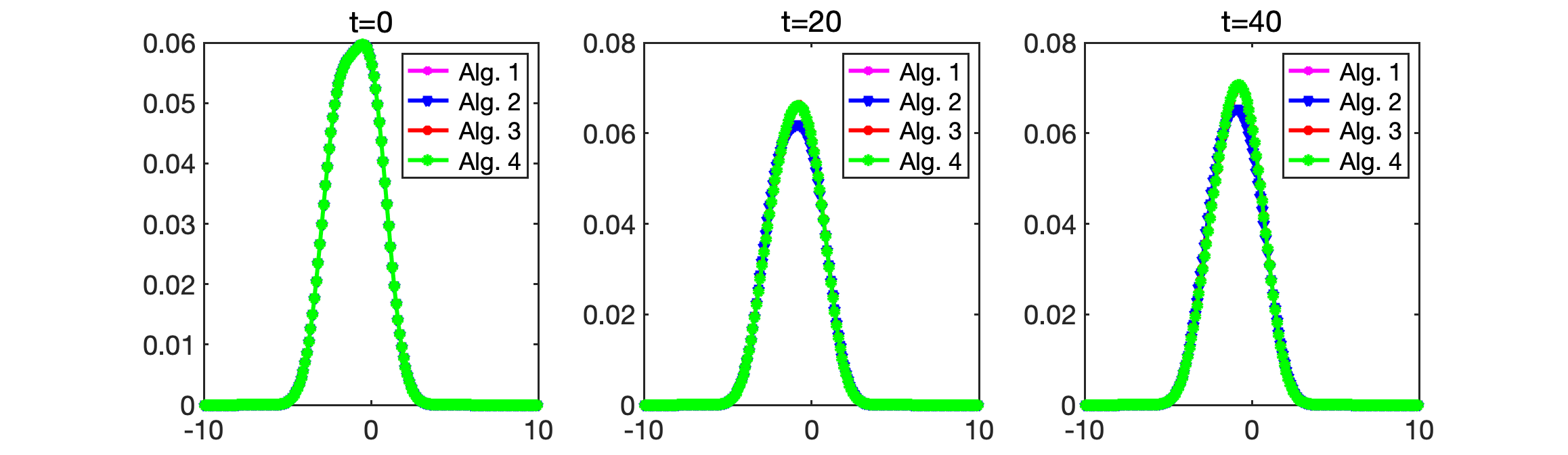}
\caption{Snapshots $f(\frac{n_o}{2},:)$ at different times when $n_o=180$.}
\label{fig:slices_t_eg3}
\end{figure}

\begin{figure}[H]
\centering
\includegraphics[width=9cm,height=3.2cm]{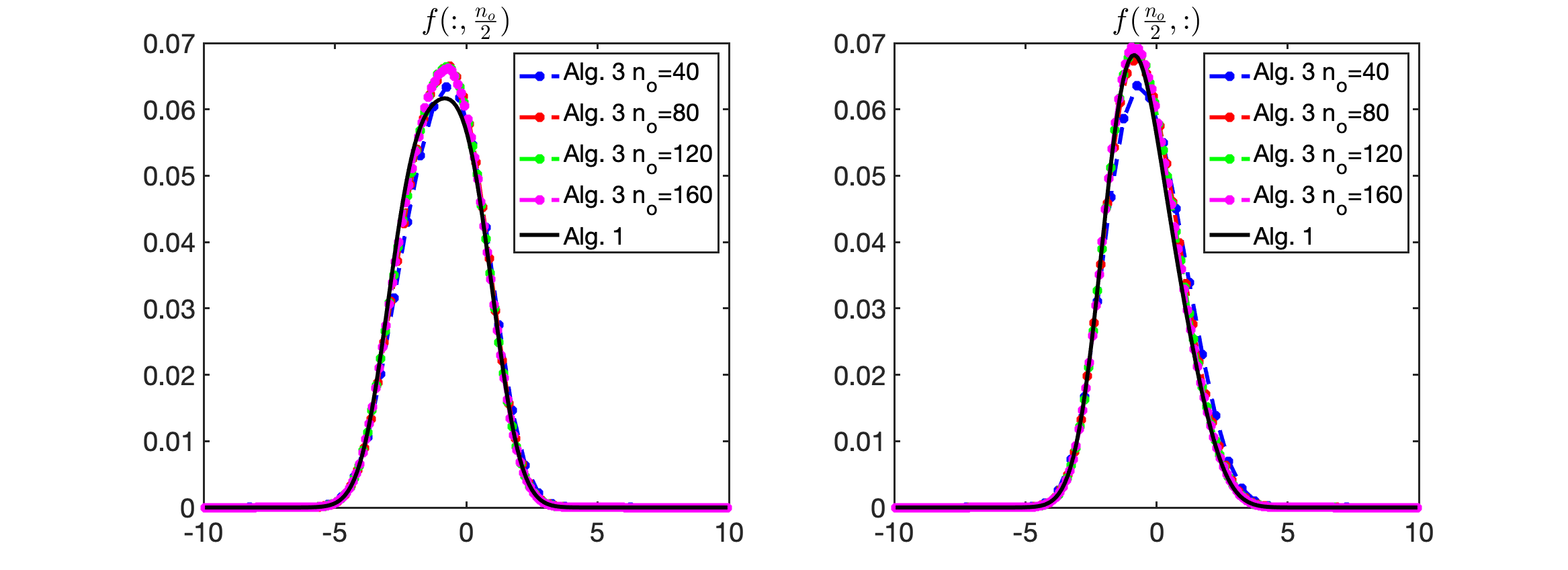}
\caption{Comparison of Alg. \ref{algo:RBM I} (using different $n_o$) with Alg. \ref{algo:I}($n_o=180$). Snapshots of the solutions at time $t=20$.}
\label{fig:slice_eg3}
\end{figure}

\begin{exmp}[3D Rosenbluth problem with Coulomb potential]
Consider the case when $d=3, \gamma=-3$, the collision kernel is $A(z)=\frac{1}{4\pi}\frac{1}{|z|^3}(|z|^2I_d-z\otimes z)$. 
The initial condition is given by
$$
f(0,v)=\frac{1}{S^2}\exp\left(-S\frac{(|v|-\mu)^2}{\mu^2}\right),\quad \mu=0.3,\quad S=10.
$$
\label{eg:4}
\end{exmp}

Let $\Delta t=0.2, L=1, q_o=4$.The cross sections $f(v_x,0,0)$ of the four algorithms at $t=0, 5, 10$ respectively are shown in Fig. \ref{fig:slice_eg4}. As time goes by, it occurs to Alg. \ref{algo:I} that the particles collide for large $n_o$. As the Coulomb kernel $A(z)$ is singular, the computation would break down once two particles collide. This can be overcome with a smaller time step. Since the homogeneous Landau equation is of diffusive type, $\Delta t=O(\Delta v^2), \Delta v=h=2L/n_o$ which is time-consuming. However, when using the random batch particle methods, the probability of two particles being close all the time is sufficiently small due to random reshuffling at each time step. So one can use a relatively bigger time step. Hence, when one wants to get the long time behavior of the Landau equation, it is preferable to use Algs. \ref{algo:RBM I} and \ref{algo:RBM II}. The cost is $O(N)$ for Algs. \ref{algo:RBM I}-\ref{algo:RBM II} as shown in Fig. \ref{fig:time_eg4}. In Fig. \ref{fig: slices_t_eg4}(left), the cross sections of the distribution function at times $t=0, 9, 36, 81, 144, 225, 900$ by Alg. \ref{algo:RBM I} are depicted. The results are in good agreement with those given in \cite{BuetCordier98,PareschiRussoToscani00}. The profiles approach the Maxwellian \eqref{eqn: Maxwellian} with 
$$\rho=\frac{2\pi\mu^3}{S^2}\left[\left(\frac{1}{2S}+1\right)\sqrt{\frac{\pi}{S}} \erfc(-\sqrt{S})+\frac{1}{S}\exp(-S)\right], \quad u=\bm{0},$$ $$T=\frac{1}{3\rho}\frac{2\pi\mu^5}{S^2}\left[ \left(1+\frac{3}{S}+\frac{3}{4S^2}\right)\sqrt{\frac{\pi}{S}} \erfc(-\sqrt{S}) +\left(\frac{1}{S}+\frac{5}{2S^2}\right)\exp(-S)\right]$$
in time as expected. Seen from Fig. \ref{fig: slices_t_eg4}, the energy and entropy are reaching the steady states as well.

\begin{figure}[H]
\centering
\includegraphics[width=14cm,height=3.9cm]{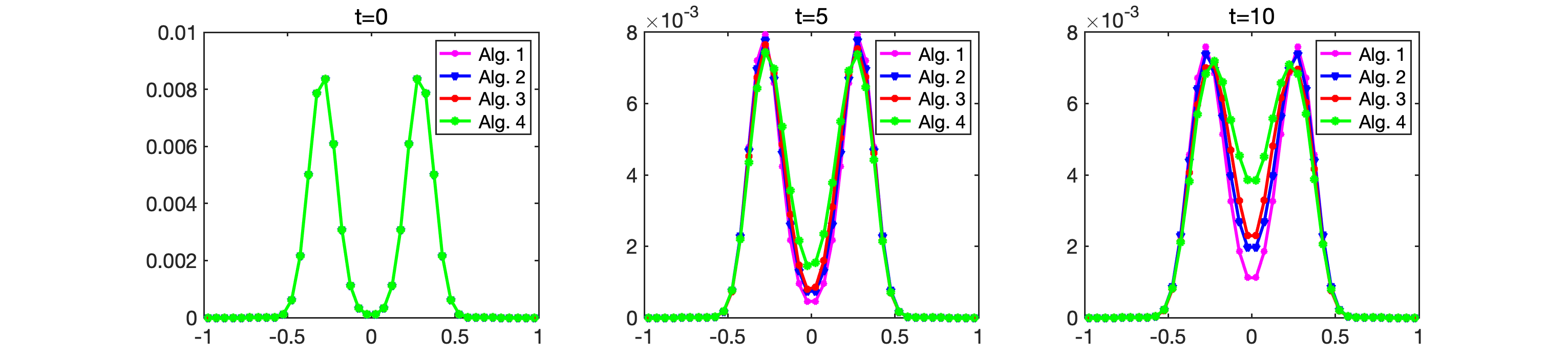}
\caption{Comparison of the four particle methods. Cross section $f(v_x,0,0)$ at different times when $n_o=40$.}
\label{fig:slice_eg4}
\end{figure}

\begin{figure}[H]
\centering
\includegraphics[width=4.5cm,height=4cm]{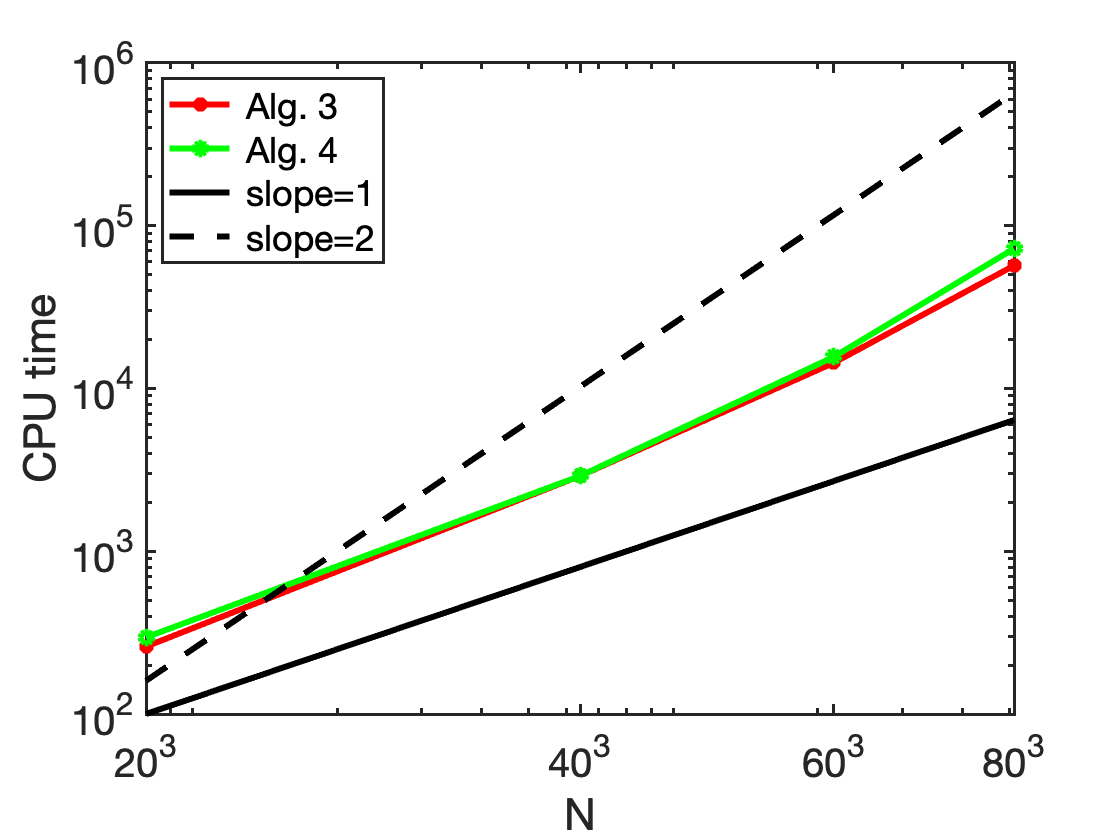}
\caption{CPU time of Algs. \ref{algo:RBM I}-\ref{algo:RBM II} w.r.t different $N$ at time $t=20$.}
\label{fig:time_eg4}
\end{figure}

\begin{figure}[H]
\centering
\includegraphics[width=4.5cm,height=4cm]{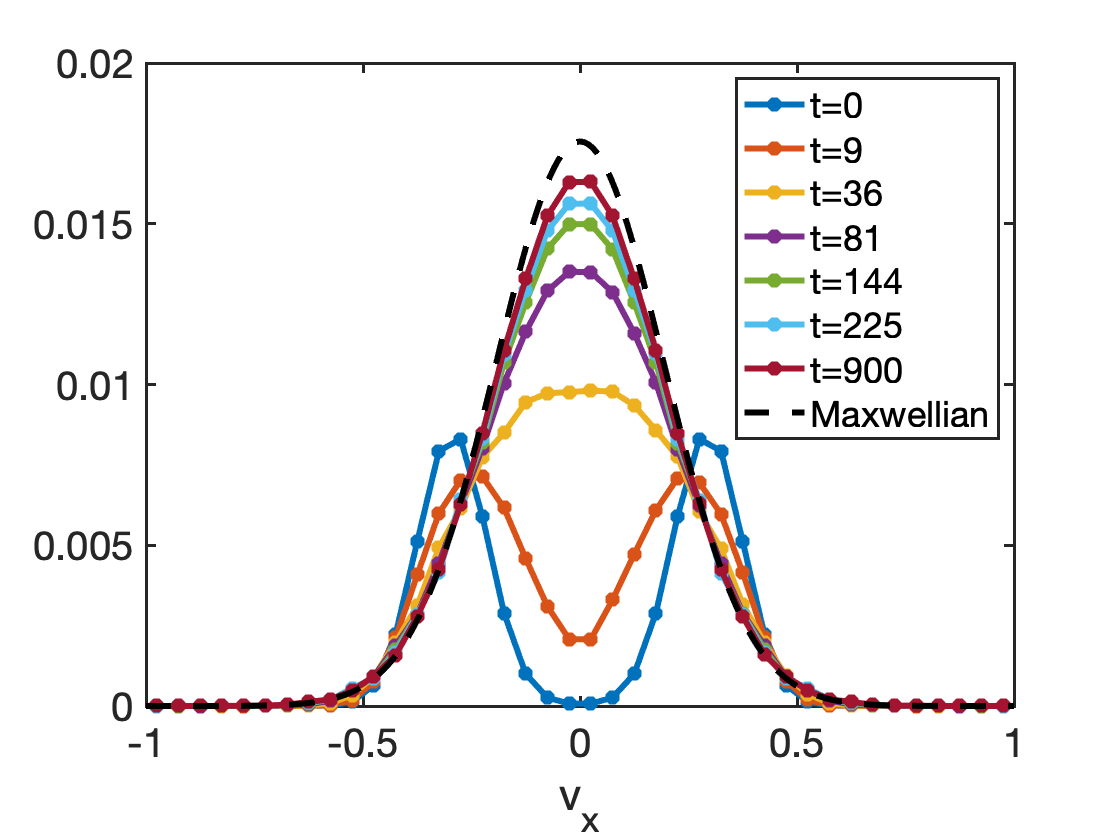}
\includegraphics[width=4.5cm,height=4cm]{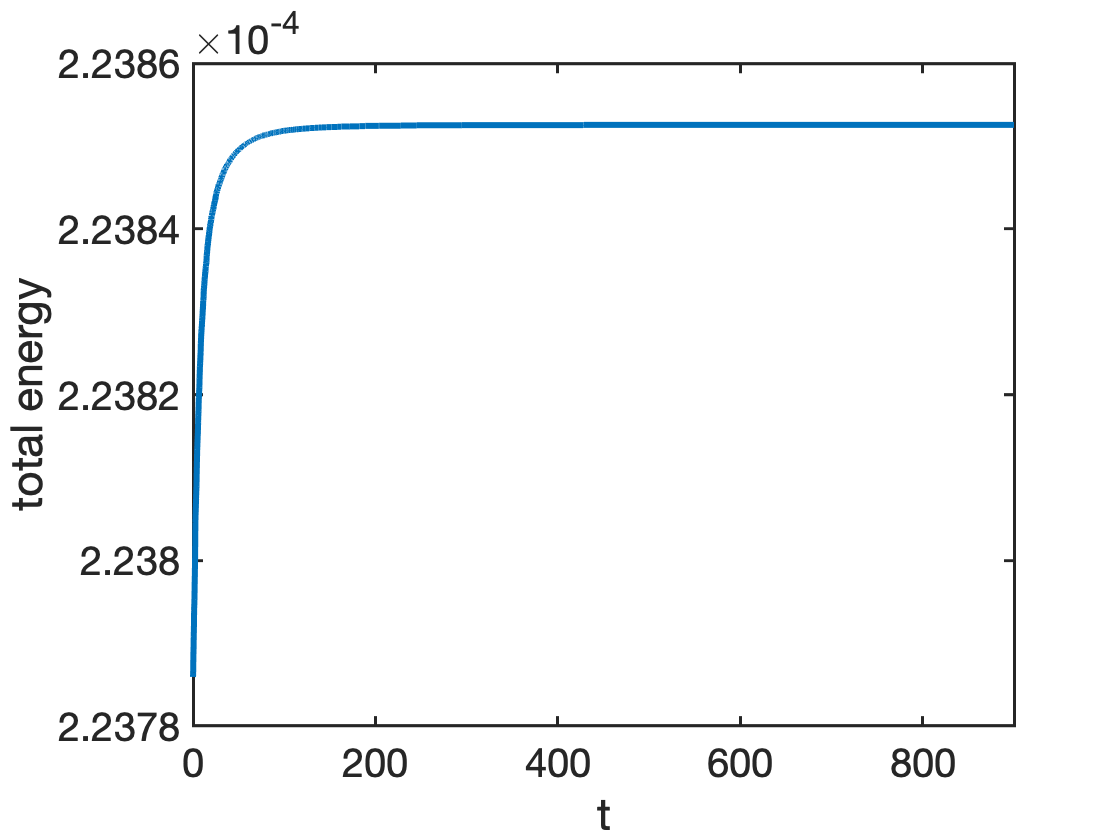}
\includegraphics[width=4.5cm,height=4cm]{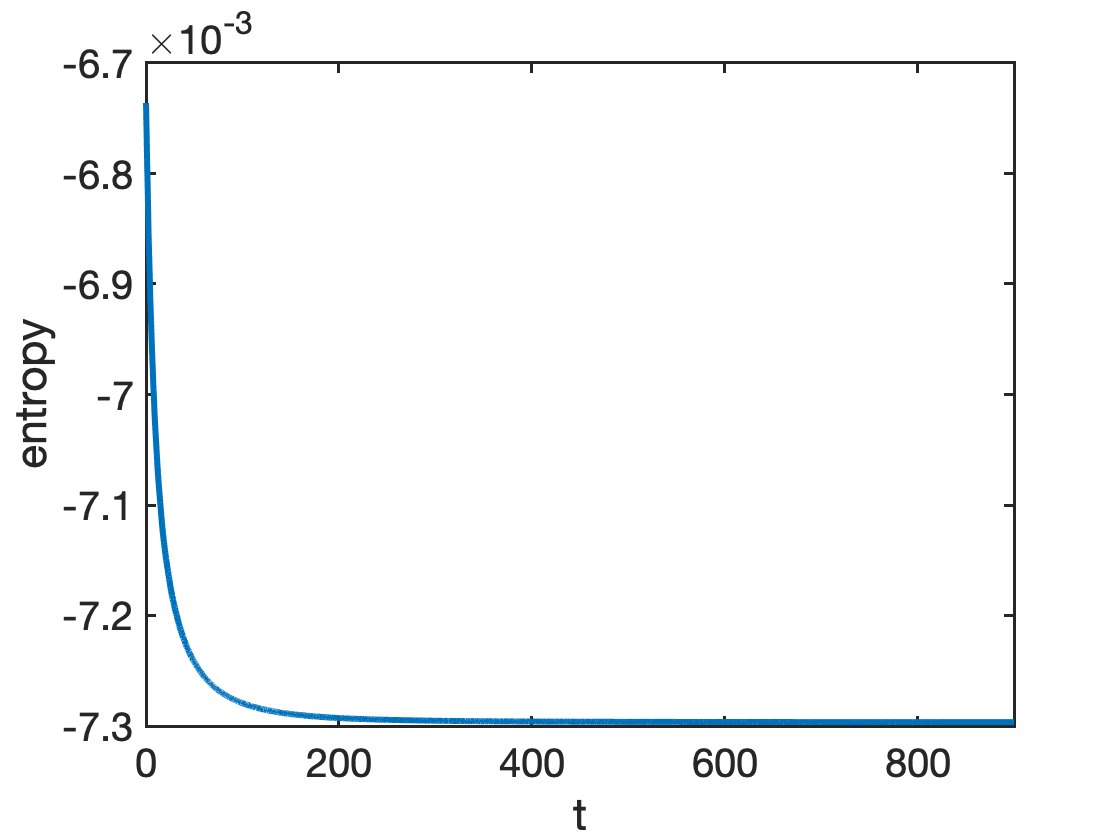}
\caption{Results by Alg. \ref{algo:RBM I} when $n_o=40$. Left: Cross section $f(v_x,0,0)$ of the distribution function at different times. The real equilibrium state is shown in the black dashed line. Middle: Time evolution of the total energy. Right: Time evolution of the entropy.}
\label{fig: slices_t_eg4}
\end{figure}

\section{Conclusion}
\label{sec: conclusion}

In this paper we introduced a random batch implementation of the particle methods for the homogeneous Landau equation proposed in \cite{CarrilloHuWangWu20}. 
For the collision term, at each time, we randomly group the $N$-particles into small batches and each particle collide only with particles in the same batch. We also utilize the rapid decay property of the mollifier, hence the overall computational cost of our algorithm is $O(N)$, instead of $O(N^2)$.
The conservation and entropy decay properties of these methods are also proved 
and numerical experiments verify the desired performance and theoretical results. 

We note that it might be a promising way to efficiently solve the Fokker-Planck-Landau equation \eqref{eqn: FPL} integrating the particle-in-cell method \cite{TskhakayaMatyashSchneiderTaccogna07} for the Vlasov equation and our random batch particle methods. This is left for our future study.

\section*{Acknowledgement}
JAC was supported by the Advanced Grant Nonlocal-CPD (Nonlocal PDEs for Complex Particle Dynamics: Phase Transitions, Patterns and Synchronization) of the European Research Council Executive Agency (ERC) under the European Union's Horizon 2020 research and innovation programme (grant agreement No. 883363).
S. Jin's research was partly supported by the NSFC grant No.12031013.

\bibliographystyle{plain}
\bibliography{rsRef.bib}

\end{document}